\documentclass[reqno]{amsart}

\usepackage{amssymb,amsfonts,amstext,amsthm,hyperref,cleveref,xcolor}
\usepackage{epstopdf}
\usepackage[caption=false]{subfig}
\usepackage{graphics,graphicx,pdfpages}
\usepackage{graphicx}
\usepackage{indentfirst}
\usepackage{float}
\usepackage[numbers,sort&compress]{natbib}
\bibpunct[, ]{[}{]}{,}{n}{,}{,}
\makeatletter
\def\NAT@def@citea{\def\@citea{\NAT@separator}}
\makeatother
\theoremstyle{plain}
\newtheorem{theorem}{Theorem}[section]

\newtheorem{proposition}[theorem]{Proposition}
\crefrangeformat{equation}{#3(#1)#4--#5(#2)#6}
\crefname{enumi}{\unskip}{\unskip}

\theoremstyle{definition}

\newtheorem{example}[theorem]{Example}

\theoremstyle{remark}
\newtheorem{remark}{Remark}

\begin{document}


\title{Involutions on  Incidence algebras of finite Posets}
\author{Ivan Gargate}
\address{UTFPR, Campus Pato Branco, Rua Via do Conhecimento km 01, 85503-390 Pato Branco, PR, Brazil}
\email{ivangargate@utfpr.edu.br}

\author{Michael Gargate}
\address{UTFPR, Campus Pato Branco, Rua Via do Conhecimento km 01, 85503-390 Pato Branco, PR, Brazil}
\email{michaelgargate@utfpr.edu.br}

\begin{abstract}
We give various formulas to compute the number of all involutions, i.e. elements of order 2, in an incidence algebra $I(X,\mathbb{K})$, where $X$ is a finite poset (star, Y and Rhombuses) and $\mathbb{K}$ is a finite field of characteristic different from 2.  Using the techniques describing here we show an algorithm to calculate the number of involutions on any finite poset.
\end{abstract}


\keywords{Involutions; incidence algebra; poset; upper triangular matrices}

\maketitle

\section{Introduction}\label{intro}

Let  $G$ be a group. An element $g\in G$ is an involution if $g^2=e$ where $e$ is the identity element. There are a lot of studies of involutions depending on the group structure, for instance, Slowik in \cite{slowik1} describe all involutions in upper triangular groups $UT_n(\mathbb{K})$. Involutions considered as automorphims on upper matrix algebras can studied in \cite{vic1} and, specially, in the extensive book \cite{merkurjev}. Involutions on graded matrix algebras were also studied in \cite{baturin1} and \cite{thiago}. Involutions as automorphisms on incidence algebras can be studied in  \cite{ros1, ros2, spi1}, and over finitary incidence algebras in \cite{ros3, ros4}.
\newline   

Slowik in \cite{slowik1} present a formula to find the number of involutions in the group of upper triangular matrices of dimension $n$, where $n\in\mathbb{N}$. 

In this article, we generalize the results of Slowik. We present various formulas to compute the number of involutions on a broader class of algebras called incidence algebras $I(X,\mathbb{K})$ defined over finite posets $X$ were $\mathbb{K}$ is  a finite field of characteristic different from  2. With the techniques shown here one can compute the number of  involutions over any finite poset. 

Our main results are the followings theorems:

\begin{theorem}\label{t1} Let $\mathbb{K}$ be a field of characteristic different from $2$ and $|\mathbb{K}|=q$. Consider the poset $X=\{x_0,x_1, x_2, \ldots, x_n,y_1, y_2,\ldots,y_m\}$ with the relations
\begin{itemize}
\item $x_0\leq x_i $ for $i=1,2, \ldots, n$, and $x_i\leq x_j$ for $i\leq j$ with $i,j=1, 2, \ldots, n$,
\item $x_0\leq y_i $ for $i=1,2, \ldots, m$, and $y_i\leq y_j$ for $i\leq j$ with $i,j=1, 2, \ldots, n$.
\end{itemize}

Denote by $\mathcal{P}(n+1+m,\mathbb{K})$ the number of elements in the incidence algebra $I(X, \mathbb{K})$, satisfying $g^2=e$. Then $\mathcal{P}(n+1+m,\mathbb{K})$ is equal to:

$$\mathcal{P}(n+1+m,\mathbb{K})=
\mathcal{I}(n+1,\mathbb{K})\times\left[  \sum_{\substack{1\leq p\leq m\\ 2|n-p}} \binom{m}{\frac{m-p}{2}} q^{\frac{1}{4}(m^2-p^2)}\left(q^{\frac{m+p}{2}}+q^{\frac{m-p}{2}}\right) +\alpha  q^{\frac{1}{4}m^2}\cdot q^{\frac{m}{2}} \right],$$ 
where $$ \alpha=\left\{\begin{array}{cl} \displaystyle\binom{m}{\frac{m}{2}} & if \ 2|m \\  & \\ 0 & if \ 2 \nmid  m\end{array}\right.
$$
and $\mathcal{I}(n+1,\mathbb{K})$ is the number of elements in the group $UT_{n+1}(\mathbb{K})$ satisfying $g^2=e$.
\end{theorem}

Using the above theorem, we obtain the general case:

\begin{theorem}\label{t2} Let $\mathbb{K}$ be a field of characteristic different from $2$ and $|\mathbb{K}|=q$. Consider the poset $X=\bigcup^{s}_{k=1}X_{k}$ such that $x_0$ is the minimal element, $\bigcap^{s}_{k=1}X_{k}=\{x_0\}$ and each  $X_{k}$ is an interval with length $m_k+1$ and denote by $\mathcal{P}(X,\mathbb{K})$ the number of involutions on the  incidence algebra $I(X,\mathbb{K})$. Define the polynomials
$P(m)$ by

$$ P(m)= \displaystyle\sum_{\substack{1\leq p\leq m \\  2|m-p}} \binom{m}{\frac{m-p}{2}} q^{\frac{1}{4}(m^2-p^2)}\left(q^{\frac{m+p}{2}}+q^{\frac{m-p}{2}}\right) + \alpha q^{\frac{1}{4}m^2}\cdot q^{\frac{m}{2}} 
$$
where $$ \alpha=\left\{\begin{array}{cl} \displaystyle\binom{m}{\frac{m}{2}} & if \ 2|m \\ &\\ 0 & if \ 2 \nmid  m\end{array}\right.
$$

Then:
$$\mathcal{P}(X,\mathbb{K})=\mathcal{I} (X_1,\mathbb{K})\cdot P(m_2)\cdot P(m_3)
\cdots P(m_s),$$
where $\mathcal{I}(X_1,\mathbb{K})$ is the number of involutions on the upper triangular matrix algebra $UT_{m_1+1}(\mathbb{K})$.
\end{theorem}
We also consider the case of Rhombuses poset: 
\begin{theorem}[The Rhombuses poset]\label{t4} Let $\mathbb{K}$ be a field of characteristic different from $2$ and $\left|\mathbb{K}\right|=q$. Consider the Rhombuses poset  $X=\{x_0,x_1,x_2, \cdots ,x_n,x_{n+1},$ $y_1$ $,y_2,$ $\cdots, y_m\}$ with the relations:
\begin{itemize}
    \item $x_0 \leq x_1 \leq \cdots \leq x_n \leq x_{n+1}$ and $ y_1 \leq y_2 \leq \cdots \leq y_m$
    \item $x_0\leq y_1$ and $y_m \leq x_{n+1}.$
    \end{itemize}
Denote by $\mathcal{P}(X,\mathbb{K})$ the number of involutions on the  incidence algebra $I(X,\mathbb{K})$, then we have
$$
\mathcal{P}(X,\mathbb{K})  =  \displaystyle 2\sum_{\substack{1\leq a\leq m\\ 2|(n-a)}}\sum_{\substack{1\leq p\leq m\\ 2|(m-b)}}\binom{n}{\frac{n-a}{2}}\binom{m}{\frac{m-b}{2}}q^{\frac{1}{4}(n^2+m^2-a^2-b^2)}\times  \mathcal{F}(a,b) + h(n,m),
$$
where
$$\mathcal{F}(a,b)=\left(q^{m+b}+q^{m-b}\right)\cdot \left(q^{n+a}+q^{n-a}\right)+4q^{n+m+1} $$
and
$$\begin{array}{ll}h(n,m)=& \alpha_1  \displaystyle\sum_{\substack{1 \leq b \leq m \\ 2| m-b}} 2\cdot \binom{n}{\frac{n}{2}}\binom{m}{\frac{m-b}{2}}\cdot q^{\frac{1}{4}(n^2+m^2-b^2)}\cdot q^{n+m}\cdot \{q^b+q^{-b}+2q\} \\
& +\alpha_2 \displaystyle\sum_{\substack{1 \leq a \leq n \\ 2|n-a}} 2 \cdot \binom{n}{\frac{n-a}{2}}\binom{m}{\frac{m}{2}} \cdot q^{\frac{1}{4}(n^2+m^2-a^2)}\cdot q^{n+m} \cdot \{ q^a+q^{-a}+2q\} \\
& + \alpha_3 \cdot \displaystyle 2 \cdot \binom{n}{\frac{n}{2}} \binom{m}{\frac{m}{2}} q^{\frac{1}{4}(n^2+m^2)}\cdot q^{n+m}\cdot (q+1)\end{array}$$
Here $\alpha_1,\alpha_2,\alpha_3 \in \{0,1\}$ and we have the following cases:
\begin{itemize}
    \item[$(i)$] If $2\nmid n$ and $2 \nmid m$ then $\alpha_1=\alpha_2=\alpha_3=0$.
    
    \item[$(ii)$] If $2| n$ and $2 \nmid m$ then $\alpha_1=1$ and $\alpha_2=\alpha_3=0$.
    
    \item[$(iii)$] If $2\nmid n$ and $2|m$ then $\alpha_2=1$ and $\alpha_1=\alpha_3=0$.
    
    \item[$(iv)$] If $2|n$ and $2|m$ then $\alpha_1=\alpha_2=\alpha_3=1$.
\end{itemize}
\end{theorem}

Finally: 
\begin{theorem}[The $Y$ Poset]\label{t3} Let $\mathbb{K}$ be a field of characteristic different from $2$ and $\left|\mathbb{K}\right|=q$. Consider the poset $Y=\{r_1,r_2, \cdots, r_n, s_1,s_2 \cdots, s_m, t_1,t_2, \cdots t_l\}$ with the relations:
\begin{itemize}
    \item $r_1,\leq r_2 \leq \cdots \leq  r_n$, $s_1 \leq s_2 \leq \cdots \leq s_m$ and $t_1 \leq t_2 \leq \cdots\leq  t_l$.
    \item $r_n \leq s_i$ with $i=1,2,\cdots, m$ and $r_n\leq t_j$ with $j=1,2,\cdots, l$.
\end{itemize}
Denote by $\mathcal{P}(X,\mathbb{K})$ the number of involutions on the  incidence algebra $I(X,\mathbb{K})$, then we have
$$
\begin{array}{rl}
\mathcal{P}(Y,\mathbb{K}) = \!\!\!\! & \! \! \! \! \displaystyle\sum^{n}_{\substack{1\leq a \leq n\\ 2| (n-a)}  } \displaystyle\sum^{m}_{\substack{1\leq b \leq m\\ 2| (m-b)}} \displaystyle\sum^{l}_{\substack{1\leq c \leq l \\ 2|(l-c)}}\binom{ n}{ \frac{n-a}{2}}\binom{m}{ \frac{m-b}{2}}\binom{l}{ \frac{l-c}{2}}q^{\frac{1}{4}(n^2+m^2+l^2-a^2-b^2-c^2)}\cdot \mathcal{F}(a,b,c)+\\
&\\
&+h(n,m,l),\end{array}
$$
where 
$$\mathcal{F}(a,b,c)=2q^{\frac{1}{2}n(m+l)}\cdot \left\{q^{\frac{1}{2}a(b+c)}+q^{\frac{1}{2}a(b-c)}+q^{\frac{1}{2}[-a(b+c)]}+q^{\frac{1}{2}[-a(b-c)]}\right\}.
$$
and
$$\begin{array}{ll}
h(n,m,l)=\! & \! \! \! \!

\alpha_1 \displaystyle\sum_{\substack{1\leq b \leq m \\ 2|(m-b)}} \ \sum_{\substack{1\leq c \leq l \\ 2|(l-c)}}\binom{n}{\frac{n}{2}}\binom{m}{\frac{m-b}{2}}\binom{l}{\frac{l-c}{2}}q^{\left\{\frac{1}{4}\left(n^2+m^2+l^2-b^2-c^2\right)\right\}}\cdot \left\{4 q^{\frac{n(m+l)}{2}}\right\} \\ &
    
   \! \! \! \! +  \alpha_2
    \displaystyle\sum_{\substack{1\leq a \leq n \\ 2|(n-a)}} \ \sum_{\substack{1\leq c \leq l \\ 2|(l-c)}}\binom{n}{\frac{n-a}{2}}\binom{m}{\frac{m}{2}}\binom{l}{\frac{l-c}{2}}q^{\left\{\frac{1}{4}\left(n^2+m^2+l^2-a^2-c^2\right)\right\}}\times  \\ & \qquad\qquad\qquad\qquad\qquad\qquad \qquad \qquad \qquad \qquad \times 2q^{\frac{nm}{2}}\cdot \left\{ q^{\frac{nl-ac}{2}}+q^{\frac{nl+ac}{2}}\right\}  \\ &
    n 
   \! \! \! \! +\alpha_3
    \displaystyle\sum_{\substack{1\leq a \leq n \\ 2|(n-a)}} \ \sum_{\substack{1\leq b \leq m \\ 2|(m-b)}}\binom{n}{\frac{n-a}{2}}\binom{m}{\frac{m-b}{2}}\binom{l}{\frac{l}{2}}q^{\left\{\frac{1}{4}\left(n^2+m^2+l^2-a^2-b^2\right)\right\}} \times \\ & \qquad\qquad\qquad\qquad\qquad\qquad \qquad \qquad \qquad \qquad \times  2q^{\frac{ln}{2}}\cdot \left\{ q^{\frac{nm-ab}{2}}+q^{\frac{nm+ab}{2}}\right\}     \\ &
    
   \! \! \! \! +\alpha_4
    \displaystyle\sum_{\substack{1\leq c \leq l \\ 2|(l-c)}} \ \binom{n}{\frac{n}{2}}\binom{m}{\frac{m}{2}}\binom{l}{\frac{l-c}{2}}q^{\left\{\frac{1}{4}\left(n^2+m^2+l^2-c^2\right)\right\}}\cdot 2  q^{\frac{n(m+l)}{2}}     \\ &

    \! \! \! \!  +\alpha_5
    \displaystyle\sum_{\substack{1\leq b \leq m \\ 2|(m-b)}} \ \binom{n}{\frac{n}{2}}\binom{m}{\frac{m-b}{2}}\binom{l}{\frac{l}{2}}q^{\left\{\frac{1}{4}\left(n^2+m^2+l^2-b^2\right)\right\}}\cdot  2q^{\frac{n(m+l)}{2}}    \\ &
    
   \! \! \! \! + \alpha_6
    \displaystyle\sum_{\substack{1\leq a \leq l \\ 2|(n-a)}} \ \binom{n}{\frac{n-a}{2}}\binom{m}{\frac{m}{2}}\binom{l}{\frac{l}{2}}q^{\left\{\frac{1}{4}\left(n^2+m^2+l^2-a^2\right)\right\}}\cdot 2 q^{\frac{n(m+l)}{2}}  \\ &
    
   \! \! \! \! +
    \alpha_7
    \displaystyle \binom{n}{\frac{n}{2}}\binom{m}{\frac{m}{2}}\binom{l}{\frac{l}{2}}q^{\left\{\frac{1}{4}\left(n^2+m^2+l^2 \right)\right\}}\cdot  q^{\frac{n(m+l)}{2}}.
\end{array}$$
Here $\alpha_1,\alpha_2,\cdots ,\alpha_7 \in \{0,1\}$ and we have the following cases:
\begin{itemize}
    \item [$(i)$] If $2\nmid n$ m $2\nmid m$ and $2\nmid l$ then $\alpha_1=\alpha_2= \cdots =\alpha_7=0$.
    \item[$(ii)$] If $2|n$, $2\nmid m$ and $2\nmid l$ then $\alpha_1=1$ and $\alpha_2,\alpha_3=\cdots =\alpha_7=0$.
    \item[$(iii)$] If $2\nmid n$, $2| m$ and $2\nmid l$ then $\alpha_2=1$ and $\alpha_1,\alpha_3=\alpha_4 \cdots =\alpha_7=0$.
    
    \item[$(iv)$] If $2\nmid n$, $2 \nmid m$ and $2| l$ then $\alpha_3=1$ and $\alpha_1,\alpha_2=\alpha_4 \cdots =\alpha_7=0$.
    
    \item[$(v)$] If $2| n$, $2| m$ and $2\nmid l$ then $\alpha_1=\alpha_2=\alpha_4=1$ and $\alpha_3=\alpha_5=\alpha_6=\alpha_7=0$.
    
    \item[$(vi)$] If $2| n$, $2\nmid m$ and $2| l$ then $\alpha_1=\alpha_3=\alpha_5=1$ and $\alpha_2=\alpha_4=\alpha_6=\alpha_7=0$.

    \item[$(vii)$] If $2\nmid n$, $2| m$ and $2 | l$ then $\alpha_2=\alpha_3=\alpha_6=1$ and $\alpha_1=\alpha_4=\alpha_5=\alpha_7=0$.
    
    \item[$(viii)$] Finally, if $2| n$, $2| m$ and $2| l$ then $\alpha_1=\alpha_2= \cdots = \alpha_7=1.$ 
    
\end{itemize}
\end{theorem}

\section{Preliminaries}

In this section, we fix the notation and recall some definitions and basic facts that will be used throughout the article.

Let $\mathbb{K}$ be a field. We denoted by $X$ a nonempty finite partially ordered set (finite poset, for short) and its relation by $\leq$  . The incidence algebra $I(X,\mathbb{K})$ of $X$ over $\mathbb{K}$ is the set $I(X,\mathbb{K})=\{ f :X\times X\longrightarrow\mathbb{K}:\  f(x, y)=0 \ \text{if}\ x\not\leq y\}$, endowed with the usual product of a map by a scalar, the
usual sum of maps, and the product defined by
$$f\cdot g(x,y)=\sum_{x\leq t\leq y}f(x,t)g(t,y),$$
for any $ f, g \in \mathcal{I}(X,\mathbb{K})$.

Let $X=\{x_1,x_2,\ldots,x_n\}$ be a finite poset, the incidence algebra $I(X,\mathbb{K})$ is a finite-dimensional linear space over $\mathbb{K}$, which has basis elements labelled $e_{ij}$ for each $i,j$ for which $x_i\leq x_j$ and has multiplication defined via

$$
e_{ij}e_{kl}=\left\{\begin{array}{cc}
e_{ij}\ &  \text{if}\ j=k\\
0\ &  \text{if}\ j\neq k
\end{array}
\right.
$$

It is not hard to see that we cant identify $I(X,\mathbb{K})$ with a sub-algebra of the full matrix algebra $M_n(\mathbb{K})$.

\begin{example}If $X=\{x_1, x_2,\ldots,x_n\}$ and the relation: $x_i\leq x_j$ if $i\leq j$. This   poset is called chain of length $n$. The Hasse diagram of this poset is

\begin{center}
	{\resizebox*{0.6cm}{!}{\includegraphics{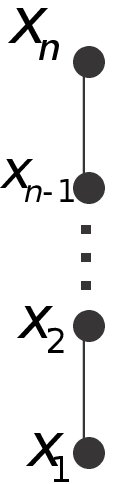}}}
	\label{graph1}
\end{center}

and  $I(X,\mathbb{K})\cong UT_n(\mathbb{K})$ is the algebra of all upper triangular matrices.
\end{example}

\begin{example}\label{example} For $X=\{x_1, x_2, x_3, x_4\}$ define the relation $x_1\leq x_3$, $x_1\leq x_4$, $x_2\leq x_3$,
$x_2\leq x_4$. This is a poset with Hasse diagram\\

\begin{figure}[H]
\centering
{\resizebox*{2.5cm}{!}{\includegraphics{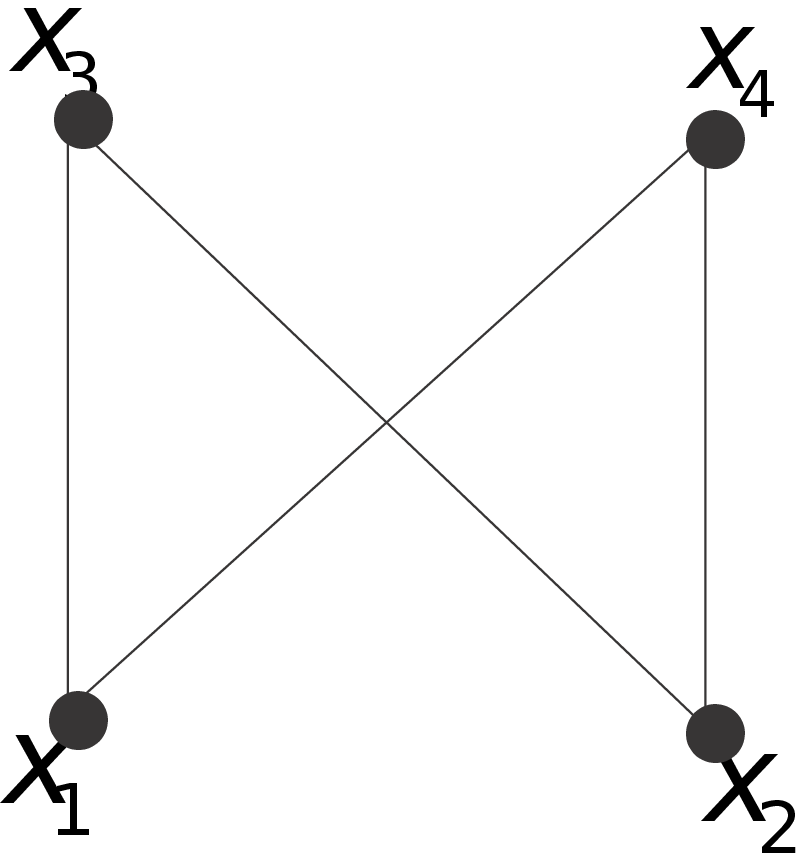}}}
 \label{graph2}
\end{figure}
and the incidence algebra $I(X,\mathbb{K})$ is the set of all
matrices of the form
$$\left(\begin{array}{cccc}
* & 0 & * & * \\
0 & * & * & * \\
0 & 0 & * & 0 \\
0 & 0 & 0 & *
\end{array}\right)
$$
\end{example}

\begin{remark} When $X$ is a finite poset, $I(X,\mathbb{K})$ will be a sub-algebra of the upper
triangular matrices (as we have seen in the last examples above), i.e. matrices of
a particular size and shape.
\end{remark}

For more facts about incidence algebras we recommend the book \cite{spi2}. Let be $g \in I(X,\mathbb{K})$, then $g$ is an involution if $g^2=e$,with $e(u,u)=1$ for all $u\in X$ and $e(u,v)=0$ if $u\neq v$. In this case we have  $g(u,u)=\pm 1$ for all $u \in X$.

Slowik in \cite{slowik1} calculated the number of involutions in the upper triangular matrices (or over incidence algebra when $X$ is a finite  chain). If fixed $d=(d_1,d_2,\cdots, d_n)$ with\\ $d_i=\pm 1$, a diagonal of an involution upper triangular matrix, then there are\\ $\Delta(d)=\frac{1}{4}[n^2-p^2]$ independent positions on this matrix, where $\left|\displaystyle\sum^n_{i=1}d_i\right|=p$, with this observation Slowik in  \cite{slowik1} proved the following:

\begin{theorem}[Slowik \cite{slowik1}]\label{teo1} Let $\mathbb{K}$ a field of characteristic different from $2$ and let $G$ the group $UT_n(\mathbb{K})$ for some $n\in \mathbb{N}$. A matrix $g \in G$ is an involution if and only if $g$ is described by the following statements:
\begin{itemize}
\item[$(i)$] For all $a\leq j\leq n$ we have $g_{ii}\in \{1,-1\}$.
\item[$(ii)$] For all pairs of indices $1\leq i \leq j \leq n$ such that $g_{ii}=-g_{jj}$, the coefficients $g_{ij}$ may be arbitrary.
\item[$(iii)$] For all pairs of indices $1\leq i\leq j \leq n$ such that $g_{ii}=g_{jj}$ we have
$$g_{ij}=\left\{\begin{array}{ll} -(2g_{ii})^{-1}\displaystyle\sum^{j-1}_{p=i+1} g_{ip}g_{pj} & if \ \ j>i+1 \\ 0 & if \ \ j=i+1 \end{array}\right.$$ 
\end{itemize}
\end{theorem}
And 

\begin{theorem}[Slowik \cite{slowik1}] Let $\mathbb{K}$ be a finite field of characteristic from $2$. If $\left|\mathbb{K}\right|=q$, then the number of elements in the group $UT_n(\mathbb{K})$, satisfying $g^2=e$, is equal to:
$$\mathcal{I}(n,\mathbb{K})=\alpha q^{\frac{n^2}{4}}+2 \displaystyle\sum_{\substack{1\leq p \leq n \\ 2 | n-p}}\displaystyle\binom{n}{\frac{n-p}{2}}q^{\frac{n^2-p^2}{4}}, \ \ where \ \ \alpha=\left\{\begin{array}{ll} \binom{n}{ \frac{n}{2}} & if \ 2| n \\ 0 & if \ 2 \not| n.\end{array}\right.$$

\end{theorem}

\section{How to count?}

Following the theorem (\ref{teo1}) we can state :
\begin{proposition}\label{propo1} Let be $g=(g_{ij}) \in UT_n(\mathbb{K})$ such that $g^2=e$. 
Then
\begin{itemize}
    \item[$(i)$] If $g_{ii}=g_{jj}$ then the variable $g_{ij}$ is $dependent$ in the sense that it is uniquely determined by values on the $i$ file and $j$ column (see theorem \ref{teo1}). We denote by $D$ the position of $g_{ij}$ in the matrix $g$.
    \item[$(ii)$] If $g_{ii}\neq g_{jj}$ then the variable $g_{ij}$ is $independent$ in the sense that the entry $g_{ij}$ can be set arbitrarily. We denote by $I$ the position of $g_{ij}$ in the matrix $g$.
\end{itemize}
\end{proposition}

\begin{example} 
Consider the matrix 
$$A=\left[\begin{array}{cc} d_1 & \theta \\ 0 & d_2 \end{array}\right]$$
such that $A^2=I_2$ with $I_2$ the identity matrix in $T_2(\mathbb{K})$, in this case:
$$A^2=\left[\begin{array}{cc} d_1^2 & \theta(d_1+d_2) \\ 0 & d_2^2 \end{array}\right]=\left[\begin{array}{cc} 1 & 0 \\ 0 & 1 \end{array}\right].$$

Denote by $d(A)=(d_1,d_2)$ the diagonal of matrix $A$, then

\begin{itemize}
\item[$(i)$] If $ d(A)=(d_1,d_2)=(1,1)$ or $d(A)=(-1,-1)$ then, by proposition \ref{propo1} $\theta$ is a  $dependent$ variable ($\theta=0$). In this case, we have only two different matrices, the identity $I_2$ and $-I_2$.
\item[$(ii)$] If $d(A)=(d_1,d_2)$ with $d_1\neq d_2$ then, by proposition \ref{propo1}, $\theta$ is an  $independent$ variable and we have $q$ different matrices with the diagonal $(1,-1)$ and $q$ different matrices with the diagonal $(-1,1)$.
\end{itemize}

In this conditions, we can represented all possibilities in the form:
$$\pm\left[\begin{array}{cc} 1 & D\\ 0 & 1 \end{array}\right],\pm \left[\begin{array}{cc} 1 & I\\ 0 & - 1 \end{array}\right].$$

In particular, we have  $2+2q$ different involutions  on $UT_2(\mathbb{K})$.
\end{example}

\begin{example}
Consider the poset $X=\{x_0,x_1,x_2,x_3\}$ with the relations $x_0\leq x_1\leq x_2$ and $x_0\leq x_3$.
The Hasse diagram of this poset is :
\begin{figure}[H]
\centering
{\resizebox*{3cm}{!}{\includegraphics{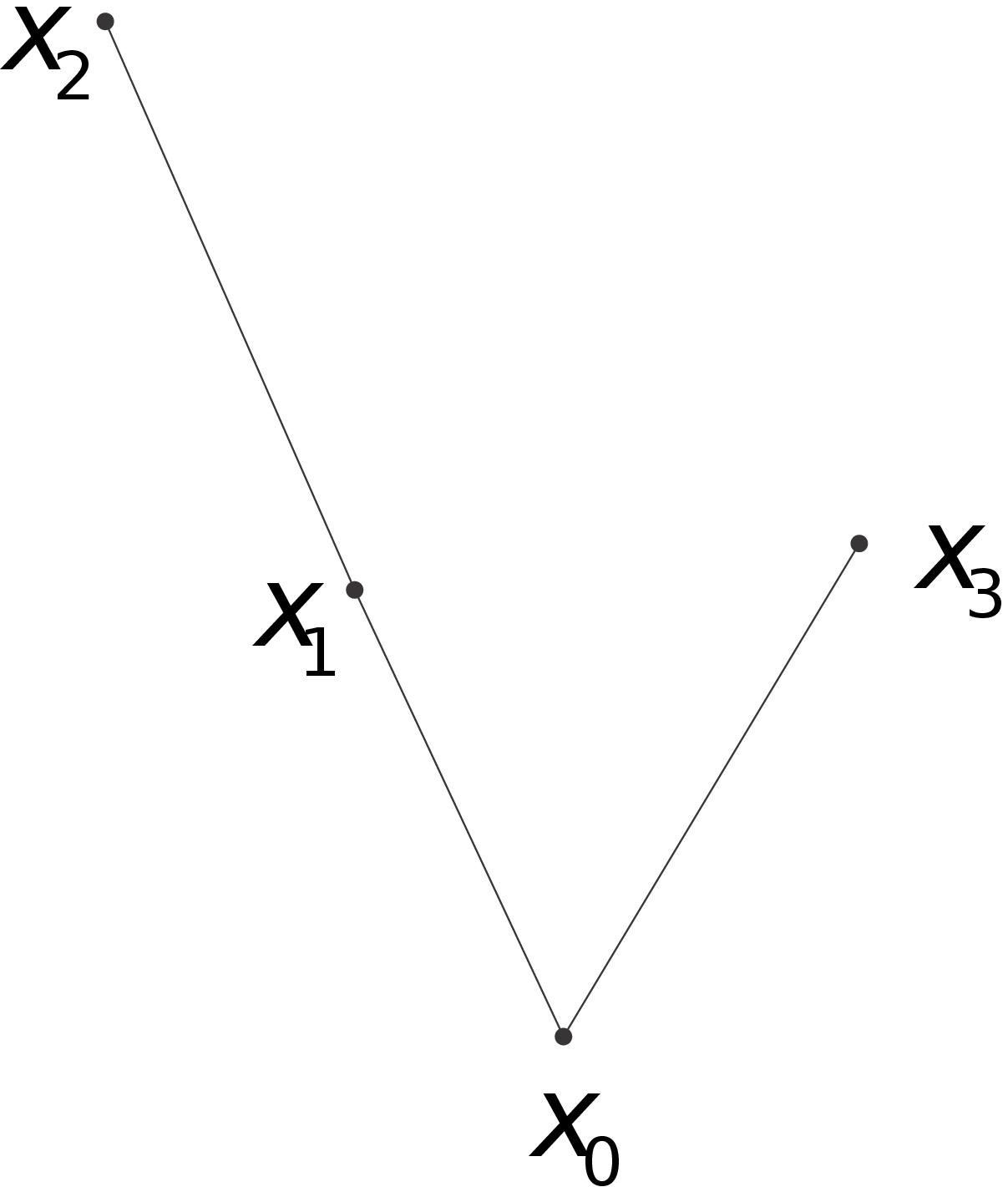}}}
\label{graph3}
\end{figure}
and the elements of their incidence algebra $I(X,\mathbb{K})$ are the form :

$$\left[\begin{array}{cccc}
    * & *& * & * \\ 0 & *&* &0  \\  0&0&*&0 \\  0&0&0&*
     
\end{array}\right]$$

Fixed the diagonal $d=(1,1,-1,-1)$ then the involutions with diagonal $d$ are in the form :

$$\left[\begin{array}{cccc}
    1 &  D & I & I\\
     0 & 1 & I & 0\\
     0& 0& -1& 0\\
     0& 0& 0& -1
\end{array}\right]$$

We can observe that there are three entries with $I$, then there are $q^3$ different involutions matrices with the fixed diagonal $d$.

\end{example} 

\begin{example}
Consider the poset of example \ref{example} If $g$ is an involution in the incidence algebra $I(X,\mathbb{K})$ then $g$ is in the form:

$$\left[\begin{array}{cc|cc}a_1 & 0 & * & * \\ 0 & a_2 & * & * \\ \hline 0 & 0& d_1 &  0 \\ 0 & 0 & 0 & d_2\end{array}\right]$$
with $a_i^2=d_j^2=1$ , $i,j=1,2$.  Then, the number of involutions can be determined by analyzing the diagonals $d_A=(a_1,a_2)$ and $d_D=(d_1,d_2)$. We observe the table:

\begin{table}[H]
\caption{Number\ involutions  with   diagonal  $d_A\times d_D$.}
\centering
\begin{tabular}{cccccc}\hline\hline
$d_A$ & & $d_D$ && & \multicolumn{1}{c}{Number of involutions} \\ [0.1ex]
\hline
$(1,1)$ && $(1,1)$ && & $1$  \\
$(-1,-1)$ && $(-1,-1)$&& & $1$  \\
$(1,-1)$ & & $(1,1)$ & && $q^2$  \\
$(-1,1)$ && $(1,1)$ && & $q^2$  \\
$(1,-1)$ && $(1,1)$ & && $q^2$  \\
$(-1,1)$ && $(-1,-1)$ && & $q^2$  \\ [1ex]
\hline
\end{tabular}
\label{table1}
\end{table}

Finally, there are $4q^2+2$ involutions in the incidence algebra $I(X,\mathbb{K})$.
\end{example}

\begin{example} Fix $m$ and $n$  odd numbers, we would like to count the number of involutions on the matrix form

$$\left[\begin{array}{c|c} R & * \\ \hline 0 & S  \end{array}\right]$$
with $R \in UT_n(\mathbb{K})$, $S \in UT_m(\mathbb{K})$ and $*$  the full matrix of order $m\times n$. Here we denote by $d_R=(d_1,d_2,\cdots, d_n)$ and $d_S=(s_1,s_2,\cdots, s_m)$  the respective diagonals. Suppose that (for $a$ and $b$ fixed) $\left|\sum^n_{i=1}d_i\right|=a$ and
$\left|\sum^m_{j=1}s_j\right|=b$, then in $d_R$:
\begin{itemize}
    \item [(i)] There are $\frac{n+a}{2}$ elements equal to $1$ and $\frac{n-a}{2}$ elements equal to $-1$ or,
    \item[(ii)] There are $\frac{n-a}{2}$ elements equal to $1$ and $\frac{n+a}{2}$ elements equal to $-1$. 
\end{itemize}
Also, in $d_S$:
\begin{itemize}
    \item [(iii)] There is $\frac{m+b}{2}$ elements equal to $1$ and $\frac{m-b}{2}$ elements equal to $-1$.
    \item[(iv)] There is $\frac{m-b}{2}$ elements equal to $1$ and $\frac{m+b}{2}$ elements equal to $-1$. 
\end{itemize}
By Slowik \cite{slowik1} we know that there is ${\Delta(d_R)}$ and ${\Delta(d_S)}$  independent entries in the matrices $R$ and $S$ respectively. The quantities of  independent entries in the full matrix $*$ depends of the quantities of products $d_i \times s_j=-1$, with $i=1,2,\cdots, n$ and $j=1,2,\cdots, m$, for example, If we consider, in $d_R$ the option (i) and in $d_S$ the option (iii) then the number os products $d_i\times s_j=-1$ is:
$$(\frac{n+a}{2})(\frac{m-b}{2})+(\frac{n-a}{2})(\frac{m+b}{2})=\frac{1}{2}(nm-ab).$$

And if one considers  in $d_R$ the option(i) and in $d_S$ the option (iv), then the number of products $d_i\times s_j=-1$ is:

$$(\frac{n+a}{2})(\frac{m+b}{2})+(\frac{n-a}{2})(\frac{m-b}{2})=\frac{1}{2}(nm+ab).$$

Then, in $*$ we have in the first situation $\frac{1}{2}(nm-ab)$  independent entries  and in the second situations we have $\frac{1}{2}(nm+ab)$  independent entries. Combining this situations with the diagonals $d_R$ and $d_S$ we have:

\begin{itemize}
    \item [(a)] If we consider (i)$\wedge$ (iii), we have
    $$\binom{n}{\frac{n-a}{2}}\binom{m}{\frac{m-b}{2}}q^{\Delta(d_T)+\Delta(d_S)}\cdot q^{\frac{1}{2}(nm-ab)},$$
    involutions, and
    \item [(b)] If we consider (i)$\wedge$(iv), we have
    $$\binom{n}{\frac{n-a}{2}}\binom{m}{\frac{m-b}{2}}q^{\Delta(d_T)+\Delta(d_S)}\cdot q^{\frac{1}{2}(nm+ab)},$$ 
involutions.
\end{itemize}

We can observe that the cases (ii) $\wedge$ (iii) and (ii) $\wedge$ (iv) are similar. 
We can summarized this results in the following figure, for example, if (i)$\wedge$ (iii) we can represented in the matrix form:

$$\begin{array}{cccc}& & & \overbrace{ \begin{tabular}{|c|c|}\hline $d_R$ & $d_S$ \\ \hline
   $\frac{n-a}{2}$ & $\frac{m-b}{2}$  \\ \hline \end{tabular}}^{Number \ of \ -1 \ in \ d_R \ and \  d_S} \\ & & &  \\   &  &
\overbrace{\begin{tabular}{|c|c|}\hline  $d_R$ & $ \frac{n+a}{2}$ \\ \hline $d_S$ & $\frac{m+b}{2}$ \\ \hline  \end{tabular}}^{Number \ of \ 1 \ in:}  &   \left[\begin{array}{c|c} \Delta(d_R) & (\frac{n+a}{2})(\frac{m-b}{2})+(\frac{n-a}{2})(\frac{m+b}{2})  \\ \hline 0 & \Delta(d_S) \end{array}\right]  \end{array}$$
where in the entries of the full matrix appears the quantities of products equal to $-1$ in $R,S$ and $*$ respectively, then in this case, see the full matrix we conclude that there are
$$\binom{n}{\frac{n-a}{2}}\binom{m}{\frac{m-b}{2}}q^{\Delta(d_R)+\Delta(d_S)+\frac{1}{2}(nm-ab),}$$
involutions.

Similarly if one considers the cases (i)$\wedge$(iv) and represented in the matrix form:

$$\begin{array}{cccc}& & & \overbrace{ \begin{tabular}{|c|c|}\hline $d_R$ & $d_S$ \\ \hline $\frac{n-a}{2}$ & $\frac{m+b}{2}$ \\ \hline   \end{tabular}}^{Number \ of \ -1 \ in:} \\ & & &  \\   &  &
\overbrace{\begin{tabular}{|c|c|} \hline $d_R$ & $ \frac{n+a}{2}$ \\ \hline $d_S$ & $\frac{m-b}{2}$ \\ \hline \end{tabular}}^{Number \ of \ 1 \ in:}  &   \left[\begin{array}{c|c} \Delta(d_R) & (\frac{n+a}{2})(\frac{m+b}{2})+(\frac{n-a}{2})(\frac{m-b}{2})  \\ \hline 0 & \Delta(d_S) \end{array}\right]  \end{array}$$
then, we conclude that there are

$$\binom{n}{\frac{n-a}{2}}\binom{m}{\frac{m-b}{2}} q^{\Delta(d_R)+\Delta(d_S)+\frac{1}{2}(nm+ab)},$$
involutions.

Finally, combining all the four cases, we conclude that there are
$$2\cdot \sum_{\substack{1\leq a \leq n \\ a \ odd}} 
\sum_{\substack{1\leq b \leq m \\ a \ odd}} \binom{n}{\frac{n-a}{2}}\binom{m}{\frac{m-b}{2}} q^{\Delta(d_R)+\Delta(d_S)}\cdot \left\{q^{\frac{1}{2}(nm+ab)}+q^{\frac{1}{2}(nm-ab)}\right\}$$ 
involutions in $UT_{n+m}(\mathbb{K})$.
\end{example}

We used this matrix representations in the proofs of theorems for minimizing counts.

\section{Proofs of results}

\begin{proof}[Proof of Theorem \ref{t1}]
The Hasse diagram of this poset is :

\begin{figure}[H]
\centering
{\resizebox*{4cm}{!}{\includegraphics{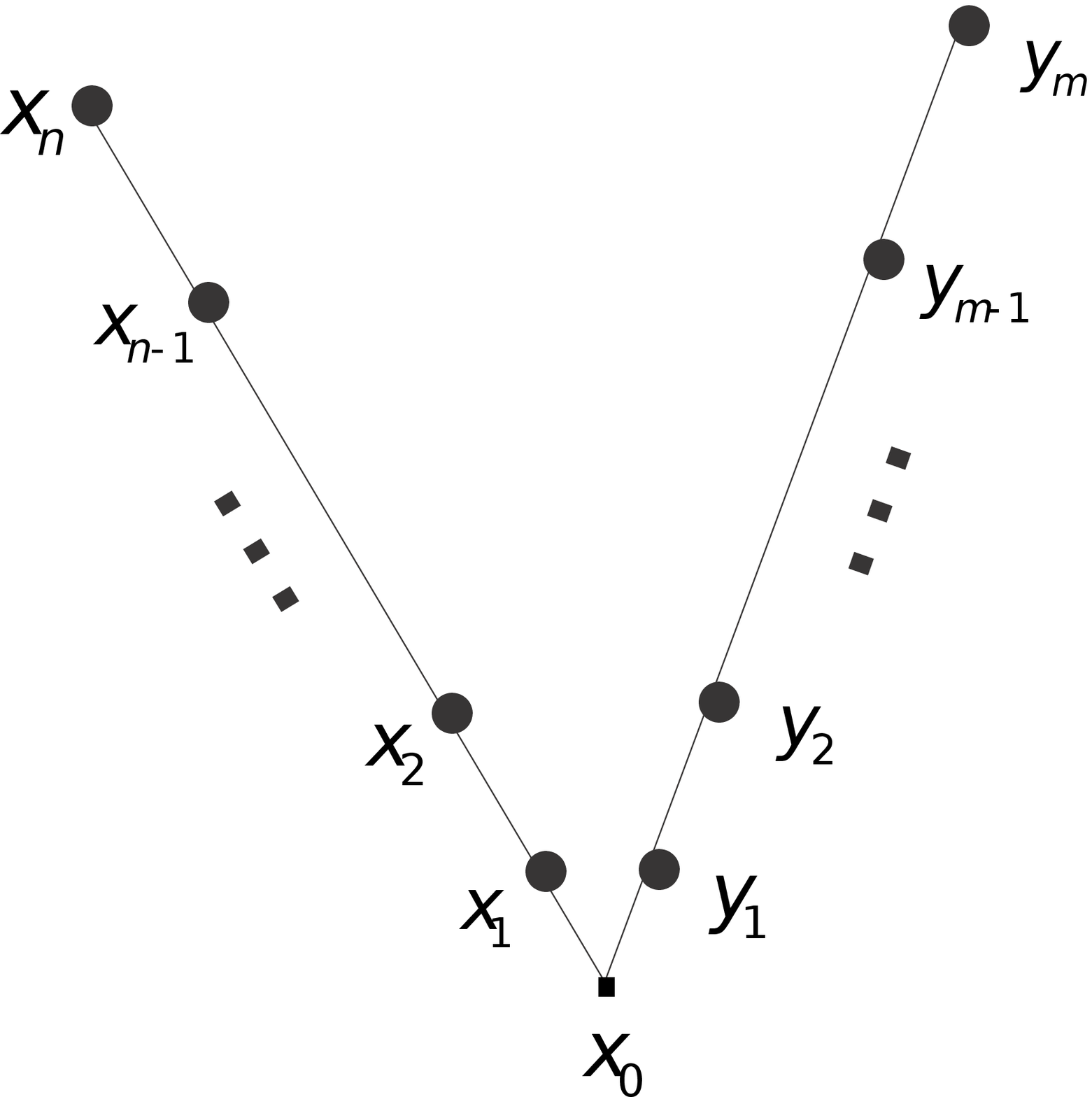}}}
 \label{graph4}
\end{figure}

and the elements of their incidence algebra $I(X,\mathbb{K})$ are the form :

$$
g=\left[\begin{array}{cccc|ccccc}
a_0 & * & * &* & \theta_1 &\theta_2 &\theta_3& \cdots  &\theta_m\\
0 & a_1 & * &* & 0  &  0 & 0  & \cdots  & 0 \\
0 & 0  & \ddots&* & 0 & 0  &  0  &  \ddots  & 0 \\
0 & 0 & 0 & a_n & 0   &  0 & 0   &  0  & 0 \\
\hline
0 & 0 & 0 & 0 & d_{11}  & *  & *   &  \cdots  & *\\
0 & 0 & 0 & 0 &  0   & d_{2}  & *   &  \cdots  &*\\
0 & 0 & 0 & 0 &  0   &  0 & d_{3}   &  \cdots  &*\\
0 & 0 & 0 & 0 &  0   &  0  & 0   &  \ddots  &  \vdots\\
0 & 0 & 0 & 0 &  0   &  0 &    &  \cdots  & d_{m}

\end{array}\right]
$$

Denote by $A=(a_{ij})$ , $D=(d_{st})$ and $\theta=(\theta_1,\theta_2,\cdots, \theta_m)$ the respective sub-matrices. Observe that, by  proposition 2.3, the dependence or independence of the entries in the sub-matrix $\theta$ depends only $a_0$ and the elements of diagonal\\  $d=(d_1,d_2,\cdots, d_m)$ in the sub-matrix $D$. Suppose that $a_0=1$ and $\left|\sum^n_{i=1}d_i\right|=p$. Then,we have two possibilities for this diagonal:
\begin{itemize}
    \item[(i)]$d$ consist of $\frac{n+p}{2}$ elements equal to $1$ and $\frac{n-p}{2}$ elements equal to $-1$, in this case we have     $\frac{n-p}{2}$ independent entries in the sub-matrix $\theta$. 
    \item[(ii)]$d$ consist of $\frac{n-p}{2}$ elements equal to $1$ and $\frac{n+p}{2}$ elements equal to $-1$, in this case, we have $\frac{n+p}{2}$ independent entries in the sub-matrix $\theta$. 
\end{itemize}

In  both case, in the sub-matrix $D$ there are $\Delta(d)$ independent entries and in the sub-matrix $A$ there are $\frac{1}{2}\mathcal{I}(n+1,\mathbb{K})$ independent entries. 

Therefore, we have
$$\frac{1}{2}\mathcal{I}(n+1,\mathbb{K})\cdot \binom{n}{ \frac{m-p}{2}}\cdot q^{\frac{1}{4}(m^2-p^2)}\cdot \{q^{\frac{m+p}{2}}+q^{\frac{m-p}{2}}\}$$
involutions with $a_0=1$ and $d$ fixed such that $\left|\sum^n_{i=1}d_i\right|=p$. 
Finally, if $m$ is odd, then $p$ is odd and hence, we have 
$$\mathcal{I}(n+1,\mathbb{K})\times\left[  \sum_{\substack{1\leq p\leq m\\  p\ odd}} \binom{m}{\frac{m-p}{2}} q^{\frac{1}{4}(m^2-p^2)}\left(q^{\frac{m+p}{2}}+q^{\frac{m-p}{2}}\right) \right] $$
involutions in $I(X,\mathbb{K})$, and if $m$ is even, then $p$ is even and we have
consider the case 
\begin{itemize}
    \item There is $\frac{m}{2}$ elements $1$ in the diagonal $d$.
\end{itemize}

In this case, we have $\Delta(A)$ independent entries in the sub-matrix $A$, $\Delta(D)=\frac{1}{4}m^2$ independent entries in the sub-matrix $D$ and $\frac{m}{2}$ independent entries in the sub-matrix $\theta$, adding this to the previous situation, we have
$$\mathcal{I}(n+1,\mathbb{K})\times\left[  \sum_{\substack{2\leq p\leq m\\ p\ even}} \binom{m}{\frac{m-p}{2}} q^{\frac{1}{4}(m^2-p^2)}\left(q^{\frac{m+p}{2}}+q^{\frac{m-p}{2}}\right) + \binom{m}{\frac{m}{2}} q^{\frac{1}{4}m^2}\cdot q^{\frac{m}{2}} \right]$$
involutions.
\end{proof}

\begin{proof}[Proof of Theorem \ref{t2}]

In this case, the Hasse diagram is
\begin{figure}[H]
\centering
{\resizebox*{5cm}{!}{\includegraphics{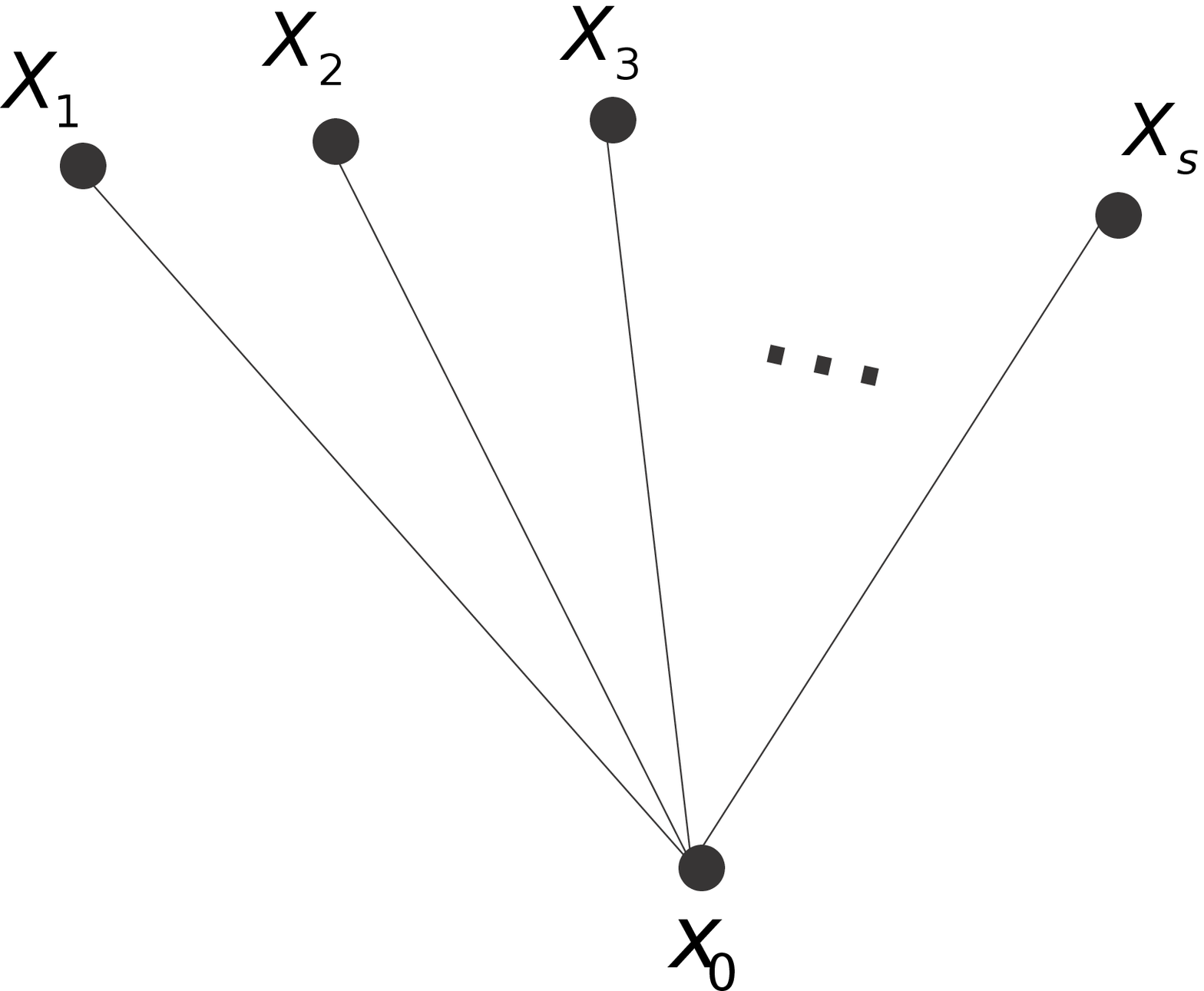}}}
\label{graph5}
\end{figure}

Suppose that we know the number of involutions for $X=\bigcup^{h}_{k=1}X_k$ and consider $Y=X \cup X_{h+1}$ with $X_{h+1}$ is an interval of length $m+1$ and $X \cap X_{h+1}=\{x_0\}$ the minimal element . Denote by $g=(g_{ij})$ a involution in the incidence algebra $I(X,\mathbb{K})$. We can observe that the involutions in the incidence algebra $I(Y,\mathbb{K})$ are in the matrix form:

$$\left[\begin{array}{c|c} g=(g_{ij}) & \begin{array}{ccccc}
\theta_1 &\theta_2 &\theta_3& \cdots  &\theta_m\\
0 & 0 & 0 & \cdots & 0 \\
\vdots & \vdots & \vdots & \ddots & \vdots \\
0 & 0 & 0 & \cdots & 0 
\end{array}    \\  

\hline 

\begin{array}{ccccc}
0 & 0 & 0 & \cdots  & 0\\
0 & 0 & 0 & \cdots & 0 \\
\vdots & \vdots & \vdots & \ddots & \vdots \\
0 & 0 & 0 & \cdots & 0 
\end{array} &

\begin{array}{ccccc}
d_1 & * & * & \cdots  & * \\
0 & d_2 & * & \cdots & * \\
\vdots & \vdots & \vdots & \ddots & \vdots \\
0 & 0 & 0 & \cdots & d_m 
\end{array}
\end{array}\right]$$

Denote by $\theta=(\theta_1,\theta_2, \cdots, \theta_m)$ and $D$ the sub-matrix with diagonal\\ $d=(d_1.d_2, \cdots d_m)$.
Fix $g_{11}=1$ and choose the diagonal $d$  such that $\left|\sum^{m}_{i=1}d_{i}\right|=p$, then there are $\Delta(D)$ independent entries in the sub-matrix $D$ and similarly to the proof of the theorem \ref{t1}, we conclude that we have
$$\frac{1}{2}\mathcal{P}(X,\mathbb{K})\cdot\binom{m}{ \frac{m-p}{2}} \cdot q^{\Delta(d)}\cdot \{q^{\frac{m+p}{2}}+q^{\frac{m-p}{2}}\}$$
involutions with $g_{11}=1$ and $\left|\sum^{m}_{i=1}d_{ii}\right|=p.$. Finally, we consider all possibilities for $p$ and conclude that there are 
$$\mathcal{P}(X,\mathbb{K})\cdot P(m)$$
involutions in $I(Y,\mathbb{K})$.
\end{proof}

\begin{proof}[Proof of Theorem \ref{t4}]
The Hasse diagram of Rhombuses poset is:

\begin{figure}[H]
\centering
{\resizebox*{4cm}{!}{\includegraphics{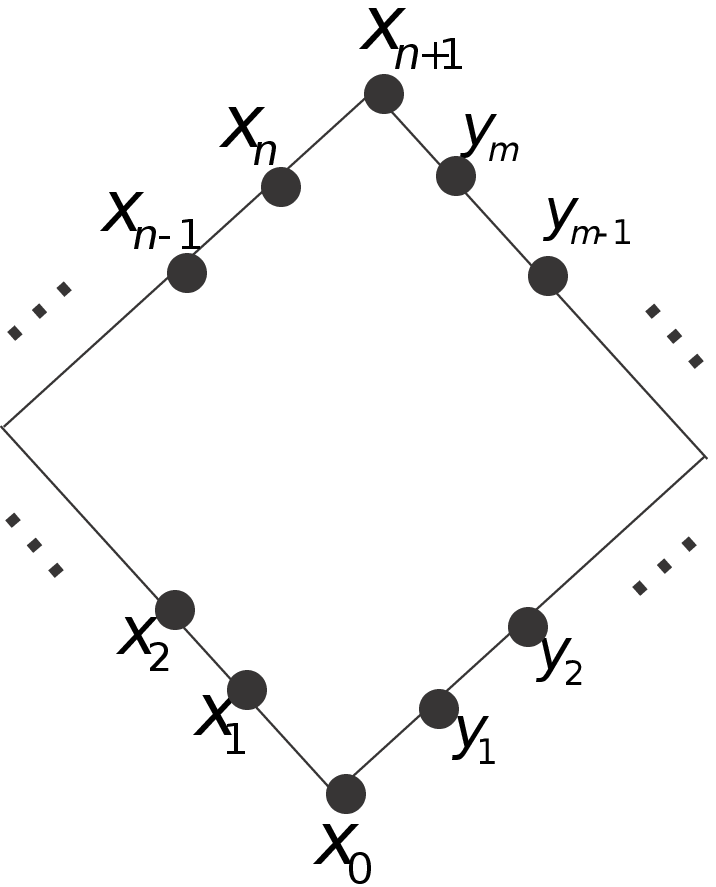}}}
\label{graph6}
\end{figure}

and the involutions elements in the incidence algebra $I(X,\mathbb{K})$ are in the form:

$$
\left[\begin{array}{c|ccc|ccccc|c}
 a_0  & \alpha_1 & \cdots &\alpha_n & \theta_1 &\theta_2 &\theta_3& \cdots  &\theta_m & \theta_{n+1}\\
\hline
0 & a_{11} & * &* & 0  &  0 & 0  & \cdots  & 0 & \beta_1 \\
0 & 0  & \ddots&* & 0 & 0  &  0  &  \ddots  & 0 &\beta_{n-1}\\
0 & 0 & 0 & a_{nn} & 0   &  0 & 0   &  0  & 0 & \beta_n\\
\hline
0 & 0 & 0 & 0 & d_{11}  & d_{12}  & d_{13}   &  \cdots  & d_{1m} & d_{1(n+1)}\\
0 & 0 & 0 & 0 &  0   & d_{22}  & d_{21}   &  \cdots  & d_{2m} & d_{2(n+1)}\\
0 & 0 & 0 & 0 &  0   &  0 & d_{33}   &  \cdots  & d_{3m}& d_{3(n+1)}\\
0 & 0 & 0 & 0 &  0   &  0  & 0   &  \ddots  &  \vdots & \vdots\\
0 & 0 & 0 & 0 &  0   &  0 &    &  \cdots  & d_{mm} & d_{m(n+1)} \\
\hline
0 & 0 & 0 & 0 & 0 & 0 & 0 & \cdots & 0 & d_{(n+1)(n+1)}

\end{array}\right]
$$

Or in block matrix form
$$\left[\begin{array}{c|c|c|c} a_0 & \alpha & \theta & \theta_{n+1} \\ \hline 0 & R & 0 & \beta \\ \hline 
0 & 0 & S & d \\ \hline 
0 & 0 & 0 & d_{(n+1)(n+1)} \end{array}\right]$$
where  $R=(a_{ik})_{i,k=1}^{n}$, $S=(d_{jl})_{j,l=1}^m$ the sub-matrices  and $d_R=(a_1,a_2,\cdots, a_n)$, $d_S=(d_{11},d_{22},\cdots d_{mm})$ the respective diagonals. Suppose that $n$ and $m$ are odd and  $\left|\sum^n_{i=1}a_i\right|=a$ and $\left|\sum^m_{j=1}d_{jj}\right|=b$, then we have the followings situations, in $d_R$:

\begin{itemize}
    \item [(i)] There is $\frac{n+a}{2}$ elements equal to $1$
 and $\frac{n-a}{2}$ elements equal to $-1$, or
    \item[(ii)] There is $\frac{n-a}{2}$ elements equal to $1$
 and $\frac{n+a}{2}$ elements equal to $-1$
 \end{itemize}
and, in $d_S$:

\begin{itemize}
    \item [(iii)] There is $\frac{m+b}{2}$ elements equal to $1$
 and $\frac{m-b}{2}$ elements equal to $-1$, or
    \item[(iv)] There is $\frac{m-b}{2}$ elements equal to $1$
 and $\frac{m+b}{2}$ elements equal to $-1$
 \end{itemize}
 
 Then, consider first that $a_0=1$, then 
 
 \begin{itemize}
    \item[(a)] If $d_{(n+1)(n+1)}=1$ then the entry correspondent to $\theta_{n+1}$ is dependent. If we have the case (i)$\wedge$(iii) then there is $\frac{n-a}{2}$ independent entries into sub-matrix $\alpha$ and $\frac{n-a}{2}$ independent entries into sub-matrix $\beta$ as well there is $\frac{m-b}{2}$ independents  entries into each sub-matrices $\theta$ and  $d$.
    In the matrix form:
    $$\begin{array}{ccccc}& & &  \overbrace{\begin{tabular}{|c|c|c|}\hline
  $d_R$  & $d_S$ & $d_{(n+1)(n+1)}$ \\ \hline
   $\frac{n-a}{2}$  & $\frac{m-b}{2}$ & $0$ \\ \hline
\end{tabular}}^{Number \ of \ -1 \ in } \\  & &  & & \\   & \overbrace{\begin{tabular}{|c|c|}\hline
  $d_R$  & $\frac{n+a}{2}$ \\ \hline  $d_S$ & $\frac{m+b}{2}$ \\ \hline
   $d_{(n+1)(n+1)}$  & $1$ \\ \hline
\end{tabular}}^{Number   \ of \  1 \ in:}  &  &  \left[\begin{array}{c|c|c|c}  a_0 & \frac{n-a}{2}  & \frac{m-b}{2} & 0  \\ \hline 0 & \Delta(d_R) & 0 & \frac{n-a}{2}    \\ \hline 0 & 0 & \Delta(d_S) &  \frac{m-b}{2} \\  \hline 0 & 0 & 0 & 0 \end{array}\right]\end{array}$$
    
    If we consider the case (i)$\wedge$(iv) then there is $\frac{m+b}{2}$  independent entries into each sub-matrices $\theta$ and $d$. In the matrix form:
    
    $$\begin{array}{ccccc}& & &  \overbrace{\begin{tabular}{|c|c|c|}\hline
  $d_R$  & $d_S$ & $d_{(n+1)(n+1)}$ \\ \hline
   $\frac{n-a}{2}$  & $\frac{m+b}{2}$ & $0$ \\ \hline
\end{tabular}}^{Number \ of \ -1 \ in } \\  & &  & & \\   & \overbrace{\begin{tabular}{|c|c|}\hline
  $d_R$  & $\frac{n+a}{2}$ \\ \hline  $d_S$ & $\frac{m-b}{2}$ \\ \hline
   $d_{(n+1)(n+1)}$  & $1$ \\ \hline
\end{tabular}}^{Number   \ of \  1 \ in:}  &  &  \left[\begin{array}{c|c|c|c}  a_0 & \frac{n-a}{2}  & \frac{m+b}{2} & 0  \\ \hline 0 & \Delta(d_R) & 0 & \frac{n-a}{2}    \\ \hline 0 & 0 & \Delta(d_S) &  \frac{m+b}{2} \\  \hline 0 & 0 & 0 & 0 \end{array}\right]\end{array}$$
     
If we consider the case (ii)$\wedge$(iii)  then we have:
     
     $$\begin{array}{ccccc}& & &  \overbrace{\begin{tabular}{|c|c|c|}\hline
  $d_R$  & $d_S$ & $d_{(n+1)(n+1)}$ \\ \hline
   $\frac{n+a}{2}$  & $\frac{m-b}{2}$ & $0$ \\ \hline
\end{tabular}}^{Number \ of \ -1 \ in } \\  & &  & & \\   & \overbrace{\begin{tabular}{|c|c|}\hline
  $d_R$  & $\frac{n-a}{2}$ \\ \hline  $d_S$ & $\frac{m+b}{2}$ \\ \hline
   $d_{(n+1)(n+1)}$  & $1$ \\ \hline
\end{tabular}}^{Number   \ of \  1 \ in:}  &  &  \left[\begin{array}{c|c|c|c}  0 & \frac{n+a}{2}  & \frac{m-b}{2} & 0  \\ \hline 0 & \Delta(d_R) & 0 & \frac{n+a}{2}    \\ \hline 0 & 0 & \Delta(d_S) &  \frac{m-b}{2} \\  \hline 0 & 0 & 0 & 0 \end{array}\right]\end{array}$$
and finally, If we consider (ii)$\wedge$(iv) then we have
    
    $$\begin{array}{ccccc}& & &  \overbrace{\begin{tabular}{|c|c|c|}\hline
  $d_R$  & $d_S$ & $d_{(n+1)(n+1)}$ \\ \hline
   $\frac{n+a}{2}$  & $\frac{m+b}{2}$ & $0$ \\ \hline
\end{tabular}}^{Number \ of \ -1 \ in } \\  & &  & & \\   & \overbrace{\begin{tabular}{|c|c|}\hline
  $d_R$  & $\frac{n-a}{2}$ \\ \hline  $d_S$ & $\frac{m-b}{2}$ \\ \hline
   $d_{(n+1)(n+1)}$  & $1$ \\ \hline
\end{tabular}}^{Number   \ of \  1 \ in:}  &  &  \left[\begin{array}{c|c|c|c}  0 & \frac{n+a}{2}  & \frac{m+b}{2} & 0  \\ \hline 0 & \Delta(d_R) & 0 & \frac{n+a}{2}    \\ \hline 0 & 0 & \Delta(d_S) &  \frac{m+b}{2} \\  \hline 0 & 0 & 0 & 0 \end{array}\right]\end{array}$$

Finally, in this case we have

$$\binom{n}{\frac{n-a}{2}}\binom{m}{\frac{m-b}{2}}q^{\Delta(d_A)+\Delta(d_D)}\cdot \left( q^{m+b}+q^{m-b}\right)\cdot \left(q^{n+a}+q^{n-a} \right),$$
involutions with diagonal $1\times d_R \times d_S \times 1$.

  \item[(b)] If $d_{(n+1)(n+1)}=-1$, then the entry correspondent to $\theta_{(n+1)}$ is independent and we have $q$ possibilities in this position. If we consider (i)$\wedge$(iii) then there is $\frac{n-a}{2}$ independent entries into sub-matrices $\alpha$ and $\beta$ and we have $\frac{n+a}{2}$ independent entries into sub-matrix $\beta$ as well  we have  $\frac{m-b}{2}$  independent entries into sub-matrix $\theta$ and $\frac{m+b}{2}$ independents entries  into sub-matrix $d$.
     In the matrix form:
    
    $$\begin{array}{ccccc}& & &  \overbrace{\begin{tabular}{|c|c|c|}\hline
  $d_R$  & $d_S$ & $d_{(n+1)(n+1)}$ \\ \hline
   $\frac{n-a}{2}$  & $\frac{m-b}{2}$ & $1$ \\ \hline
\end{tabular}}^{Number \ of \ -1 \ in } \\  & &  & & \\   & \overbrace{\begin{tabular}{|c|c|}\hline
  $d_R$  & $\frac{n+a}{2}$ \\ \hline  $d_S$ & $\frac{m+b}{2}$ \\ \hline
   $d_{(n+1)(n+1)}$  & $0$ \\ \hline
\end{tabular}}^{Number   \ of \  1 \ in:}  &  &  \left[\begin{array}{c|c|c|c}  a_0 & \frac{n-a}{2}  & \frac{m-b}{2} & 1 \\ \hline 0 & \Delta(d_R) & 0 & \frac{n-a}{2}    \\ \hline 0 & 0 & \Delta(d_S) &  \frac{m+b}{2} \\  \hline 0 & 0 & 0 & 0 \end{array}\right]\end{array}$$

    If we consider the case (ii)$\wedge$(iii)  then there are $\frac{n+a}{2}$ independent entries into sub-matrix $\alpha$ and $\frac{n-a}{2}$ independent entries into sub-matrix $\beta$ as well we have $\frac{m-b}{2}$ into sub-matrix $\theta$ and $\frac{m+b}{2}$ independent entries into sub-matrix $d$. In the matrix form:
    
     $$\begin{array}{ccccc}& & &  \overbrace{\begin{tabular}{|c|c|c|}\hline
  $d_R$  & $d_S$ & $d_{(n+1)(n+1)}$ \\ \hline
   $\frac{n+a}{2}$  & $\frac{m-b}{2}$ & $1$ \\ \hline
\end{tabular}}^{Number \ of \ -1 \ in } \\  & &  & & \\   & \overbrace{\begin{tabular}{|c|c|}\hline
  $d_R$  & $\frac{n-a}{2}$ \\ \hline  $d_S$ & $\frac{m+b}{2}$ \\ \hline
   $d_{(n+1)(n+1)}$  & $0$ \\ \hline
\end{tabular}}^{Number   \ of \  1 \ in:}  &  &  \left[\begin{array}{c|c|c|c}  0 & \frac{n+a}{2}  & \frac{m-b}{2} & 1 \\ \hline 0 & \Delta(d_R) & 0 & \frac{n-a}{2}    \\ \hline 0 & 0 & \Delta(d_S) &  \frac{m+b}{2} \\  \hline 0 & 0 & 0 & 0 \end{array}\right]\end{array}$$

If we consider (i)$\wedge$(iv) then we have the matrix form:

$$\begin{array}{ccccc}& & &  \overbrace{\begin{tabular}{|c|c|c|}\hline
  $d_R$  & $d_S$ & $d_{(n+1)(n+1)}$ \\ \hline
   $\frac{n-a}{2}$  & $\frac{m+b}{2}$ & $1$ \\ \hline
\end{tabular}}^{Number \ of \ -1 \ in } \\  & &  & & \\   & \overbrace{\begin{tabular}{|c|c|}\hline
  $d_R$  & $\frac{n+a}{2}$ \\ \hline  $d_S$ & $\frac{m-b}{2}$ \\ \hline
   $d_{(n+1)(n+1)}$  & $0$ \\ \hline
\end{tabular}}^{Number   \ of \  1 \ in:}  &  &  \left[\begin{array}{c|c|c|c}  a_0 & \frac{n-a}{2}  & \frac{m+b}{2} & 1 \\ \hline 0 & \Delta(d_R) & 0 & \frac{n-a}{2}    \\ \hline 0 & 0 & \Delta(d_S) &  \frac{m-b}{2} \\  \hline 0 & 0 & 0 & 0 \end{array}\right]\end{array}$$

and finally, consider (ii)$\wedge$(iv) we have the matrix form:

  $$\begin{array}{ccccc}& & &  \overbrace{\begin{tabular}{|c|c|c|}\hline
  $d_R$  & $d_S$ & $d_{(n+1)(n+1)}$ \\ \hline
   $\frac{n+a}{2}$  & $\frac{m+b}{2}$ & $1$ \\ \hline
\end{tabular}}^{Number \ of \ -1 \ in } \\  & &  & & \\   & \overbrace{\begin{tabular}{|c|c|}\hline
  $d_R$  & $\frac{n-a}{2}$ \\ \hline  $d_S$ & $\frac{m-b}{2}$ \\ \hline
   $d_{(n+1)(n+1)}$  & $0$ \\ \hline
\end{tabular}}^{Number   \ of \  1 \ in:}  &  &  \left[\begin{array}{c|c|c|c}  0 & \frac{n+a}{2}  & \frac{m+b}{2} & 1 \\ \hline 0 & \Delta(d_R) & 0 & \frac{n-a}{2}    \\ \hline 0 & 0 & \Delta(d_S) &  \frac{m-b}{2} \\  \hline 0 & 0 & 0 & 0 \end{array}\right]\end{array}$$
    \end{itemize}

If $a_0=-1$, we can observe that we have the all similar cases, combining all we conclude that:

 $$\mathcal{P}(X,\mathbb{K})=\displaystyle 2\sum_{\substack{1\leq a \leq n \\ a \ odd}}\sum_{\substack{1 \leq b \leq m \\ m \ odd}}\binom{n}{\frac{n-a}{2}}\binom{m}{\frac{m-b}{2}}q^{\Delta(d_R)+\Delta(d_S)}\cdot \mathcal{F}(a,b)
,$$
with $$\mathcal{F}(a,b)=\left(q^{m+b}+q^{m-b}\right)\cdot \left(q^{n+a}+q^{n-a}\right)+4q^{n+m+1}.$$
 
In the case, for example, that $n$ is even and $m$ is odd we have add the situation
\begin{itemize}
    \item [(A)] There is $\frac{n}{2}$ elements $1$ in $d_R$.
\end{itemize}

We have consider the cases : \begin{itemize}
    \item [(a)] ($a_0=1$)$\wedge$(A)$\wedge$(iii)$\wedge$($d_{(n+1)(n+1)}=1$) then, their matrix form is:

$$\begin{array}{ccccc}& & &  \overbrace{\begin{tabular}{|c|c|c|}\hline
  $d_R$  & $d_S$ & $d_{(n+1)(n+1)}$ \\ \hline
   $\frac{n}{2}$  & $\frac{m-b}{2}$ & $0$ \\ \hline
\end{tabular}}^{Number \ of \ -1 \ in } \\  & &  & & \\   & \overbrace{\begin{tabular}{|c|c|}\hline
  $d_R$  & $\frac{n}{2}$ \\ \hline  $d_S$ & $\frac{m+b}{2}$ \\ \hline
   $d_{(n+1)(n+1)}$  & $1$ \\ \hline
\end{tabular}}^{Number   \ of \  1 \ in:}  &  &  \left[\begin{array}{c|c|c|c}  0 & \frac{n}{2}  & \frac{m-b}{2} & 1 \\ \hline 0 & \Delta(d_R) & 0 & \frac{n}{2}    \\ \hline 0 & 0 & \Delta(d_S) &  \frac{m-b}{2} \\  \hline 0 & 0 & 0 & 0 \end{array}\right]\end{array}$$

\item[(b)] ($a_0=1$)$\wedge$(A)$\wedge$(iv)$\wedge$($d_{(n+1)(n+1)}=1$) then

$$\begin{array}{ccccc}& & &  \overbrace{\begin{tabular}{|c|c|c|}\hline
  $d_R$  & $d_S$ & $d_{(n+1)(n+1)}$ \\ \hline
   $\frac{n}{2}$  & $\frac{m+b}{2}$ & $0$ \\ \hline
\end{tabular}}^{Number \ of \ -1 \ in } \\  & &  & & \\   & \overbrace{\begin{tabular}{|c|c|}\hline
  $d_R$  & $\frac{n}{2}$ \\ \hline  $d_S$ & $\frac{m-b}{2}$ \\ \hline
   $d_{(n+1)(n+1)}$  & $1$ \\ \hline
\end{tabular}}^{Number   \ of \  1 \ in:}  &  &  \left[\begin{array}{c|c|c|c}  0 & \frac{n}{2}  & \frac{m+b}{2} & 1 \\ \hline 0 & \Delta(d_R) & 0 & \frac{n}{2}    \\ \hline 0 & 0 & \Delta(d_S) &  \frac{m+b}{2} \\  \hline 0 & 0 & 0 & 0 \end{array}\right]\end{array}$$

\item[(c)] ($a_0=1$)$\wedge$(A)$\wedge$(iii)$\wedge$($d_{(n+1)(n+1)}=-1$) then

$$\begin{array}{ccccc}& & &  \overbrace{\begin{tabular}{|c|c|c|}\hline
  $d_R$  & $d_S$ & $d_{(n+1)(n+1)}$ \\ \hline
   $\frac{n}{2}$  & $\frac{m-b}{2}$ & $1$ \\ \hline
\end{tabular}}^{Number \ of \ -1 \ in } \\  & &  & & \\   & \overbrace{\begin{tabular}{|c|c|}\hline
  $d_R$  & $\frac{n}{2}$ \\ \hline  $d_S$ & $\frac{m+b}{2}$ \\ \hline
   $d_{(n+1)(n+1)}$  & $0$ \\ \hline
\end{tabular}}^{Number   \ of \  1 \ in:}  &  &  \left[\begin{array}{c|c|c|c}  0 & \frac{n}{2}  & \frac{m-b}{2} & 1 \\ \hline 0 & \Delta(d_R) & 0 & \frac{n}{2}    \\ \hline 0 & 0 & \Delta(d_S) &  \frac{m-b}{2} \\  \hline 0 & 0 & 0 & 0 \end{array}\right]\end{array}$$

\item[(d)] ($a_0=1$)$\wedge$(A)$\wedge$(iv)$\wedge$($d_{(n+1)(n+1)}=-1$) then

$$\begin{array}{ccccc}& & &  \overbrace{\begin{tabular}{|c|c|c|}\hline
  $d_R$  & $d_S$ & $d_{(n+1)(n+1)}$ \\ \hline
   $\frac{n}{2}$  & $\frac{m+b}{2}$ & $1$ \\ \hline
\end{tabular}}^{Number \ of \ -1 \ in } \\  & &  & & \\   & \overbrace{\begin{tabular}{|c|c|}\hline
  $d_R$  & $\frac{n}{2}$ \\ \hline  $d_S$ & $\frac{m-b}{2}$ \\ \hline
   $d_{(n+1)(n+1)}$  & $0$ \\ \hline
\end{tabular}}^{Number   \ of \  1 \ in:}  &  &  \left[\begin{array}{c|c|c|c}  0 & \frac{n}{2}  & \frac{m+b}{2} & 1 \\ \hline 0 & \Delta(d_R) & 0 & \frac{n}{2}    \\ \hline 0 & 0 & \Delta(d_S) &  \frac{m-b}{2} \\  \hline 0 & 0 & 0 & 0 \end{array}\right]\end{array}$$

\end{itemize}
Then, combining this cases we have add the term:

$$h(n,m)=  \displaystyle\sum_{\substack{1 \leq b \leq m \\ 2| m-b}} 2\cdot \binom{n}{\frac{n}{2}}\binom{m}{\frac{m-b}{2}}\cdot q^{\Delta(d_R)+\Delta(d_S)}\cdot q^{n+m}\cdot \{q^b+q^{-b}+2q\}.$$

The case $m$ even and $n$ odd is similar and we have include the following situation:
\begin{itemize}
    \item [(B)] There is $\frac{m}{2}$ elements $1$ in the diagonal $d_S$.
\end{itemize}

In the cases that all are odd number, we have consider the following situations 
\begin{itemize}
    \item ($a_0=1$)$\wedge$ (A) $\wedge$ (iii)$\wedge$ ($d_{(n+1)(n+1)}=1$),
    \item ($a_0=1$)$\wedge$ (A) $\wedge$ (iii)$\wedge$ ($d_{(n+1)(n+1)}=-1$),
    \item ($a_0=1$)$\wedge$ (A) $\wedge$ (iv)$\wedge$ ($d_{(n+1)(n+1)}=1$),
    \item ($a_0=1$)$\wedge$ (A) $\wedge$ (iv)$\wedge$ ($d_{(n+1)(n+1)}=-1$),
    \item ($a_0=1$)$\wedge$ (i) $\wedge$ (B)$\wedge$ ($d_{(n+1)(n+1)}=1$),
    \item
($a_0=1$)$\wedge$ (i) $\wedge$ (B)$\wedge$ ($d_{(n+1)(n+1)}=-1$),
\item
($a_0=1$)$\wedge$ (ii) $\wedge$ (B)$\wedge$ ($d_{(n+1)(n+1)}=1$),
\item
($a_0=1$)$\wedge$ (ii) $\wedge$ (B)$\wedge$ ($d_{(n+1)(n+1)}=-1$),
\item
($a_0=1$)$\wedge$ (A) $\wedge$ (B)$\wedge$ ($d_{(n+1)(n+1)}=1$)
\item 
and finally 
($a_0=1$)$\wedge$ (A) $\wedge$ (B)$\wedge$ ($d_{(n+1)(n+1)}=-1$).
\end{itemize}
  
 We analyzing the last two situations:
 \begin{itemize}
     \item [(1)] If ($a_0=1$)$\wedge$ (A) $\wedge$ (B)$\wedge$ ($d_{(n+1)(n+1)}=1$) the we have the following matrix representation:
     
  $$\begin{array}{ccccc}& & &  \overbrace{\begin{tabular}{|c|c|c|}\hline
  $d_R$  & $d_S$ & $d_{(n+1)(n+1)}$ \\ \hline
   $\frac{n}{2}$  & $\frac{m}{2}$ & $0$ \\ \hline
\end{tabular}}^{Number \ of \ -1 \ in } \\  & &  & & \\   & \overbrace{\begin{tabular}{|c|c|}\hline
  $d_R$  & $\frac{n}{2}$ \\ \hline  $d_S$ & $\frac{m}{2}$ \\ \hline
   $d_{(n+1)(n+1)}$  & $1$ \\ \hline
\end{tabular}}^{Number   \ of \  1 \ in:}  &  &  \left[\begin{array}{c|c|c|c}  0 & \frac{n}{2}  & \frac{m}{2} & 0 \\ \hline 0 & \Delta(d_R) & 0 & \frac{n}{2}    \\ \hline 0 & 0 & \Delta(d_S) &  \frac{m}{2} \\  \hline 0 & 0 & 0 & 0 \end{array}\right]\end{array}$$   
     
     \item[(2)] If ($a_0=1$)$\wedge$ (A) $\wedge$ (B)$\wedge$ ($d_{(n+1)(n+1)}=-1$) then we have
    $$\begin{array}{ccccc}& & &  \overbrace{\begin{tabular}{|c|c|c|}\hline
  $d_R$  & $d_S$ & $d_{(n+1)(n+1)}$ \\ \hline
   $\frac{n}{2}$  & $\frac{m}{2}$ & $1$ \\ \hline
\end{tabular}}^{Number \ of \ -1 \ in } \\  & &  & & \\   & \overbrace{\begin{tabular}{|c|c|}\hline
  $d_R$  & $\frac{n}{2}$ \\ \hline  $d_S$ & $\frac{m}{2}$ \\ \hline
   $d_{(n+1)(n+1)}$  & $0$ \\ \hline
\end{tabular}}^{Number   \ of \  1 \ in:}  &  &  \left[\begin{array}{c|c|c|c}  0 & \frac{n}{2}  & \frac{m}{2} & 1 \\ \hline 0 & \Delta(d_R) & 0 & \frac{n}{2}    \\ \hline 0 & 0 & \Delta(d_S) &  \frac{m}{2} \\  \hline 0 & 0 & 0 & 0 \end{array}\right]\end{array}$$

 \end{itemize}
 And finally, combining all we have add the term 
 
 $$h(n,m)= \cdots +\displaystyle 2 \cdot \binom{n}{\frac{n}{2}} \binom{m}{\frac{m}{2}} q^{\frac{1}{4}(n^2+m^2)}\cdot q^{n+m}\cdot (q+1).$$
  
\end{proof}

\begin{proof}[Proof of Theorem \ref{t3}] The Hasse diagram for this poset is
\begin{figure}[H]
\centering
{\resizebox*{4cm}{!}{\includegraphics{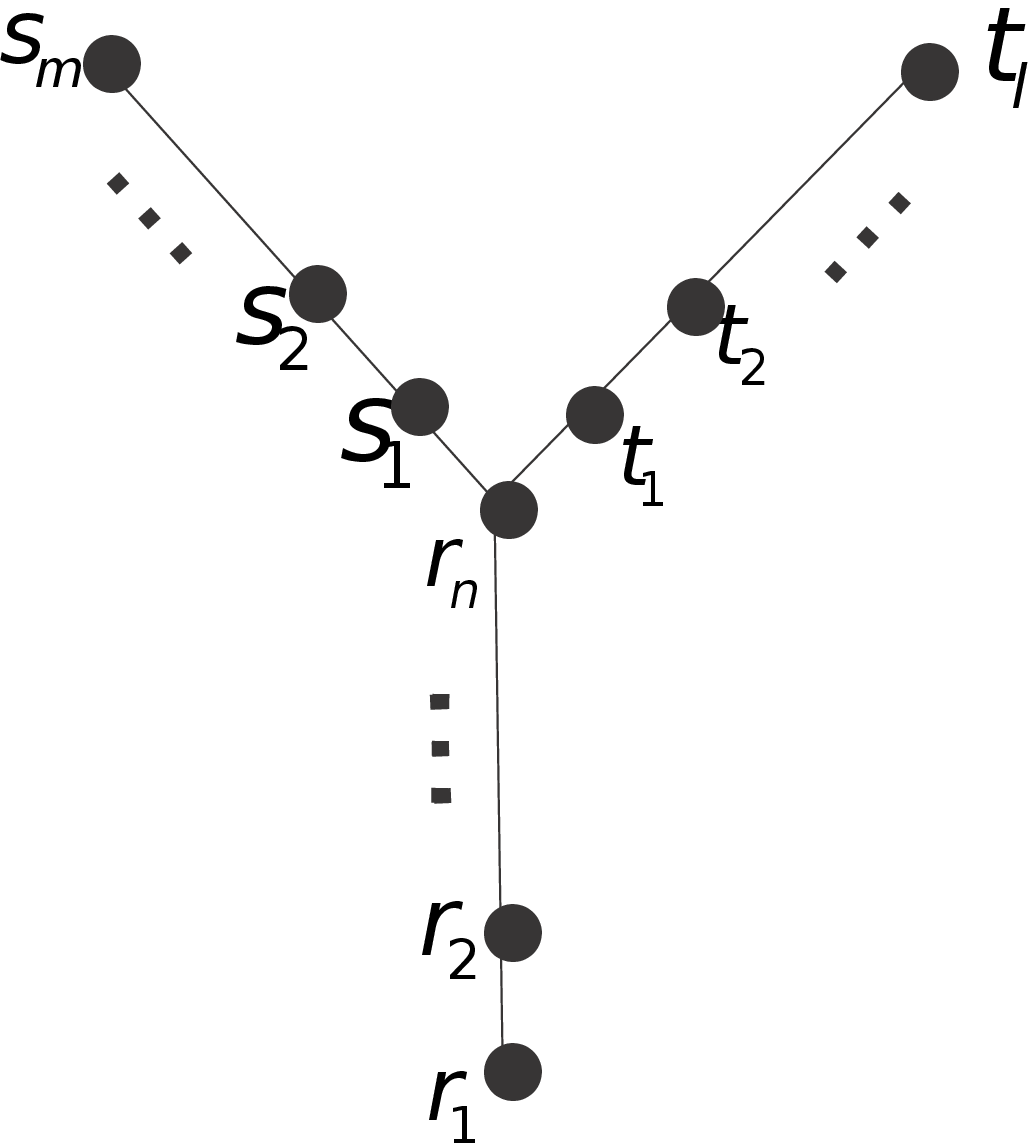}}}
\label{graph7}
\end{figure}
The elements in the incidence algebra $I(Y,\mathbb{K})$ are in the matrix form:

$$\left[\begin{array}{c|c|c} R & *_1 & *_2 \\ \hline 0 & S & 0 \\ \hline 0 & 0 & T\end{array}\right]$$
with, $R, S$ and $T$ are triangular superior matrix and $*_1$ and $*_2$ are a full matrix. Denote by $d_R=(\alpha_1,\alpha_2, \cdots, \alpha_n), d_S=(\beta_1,\beta_2,\cdots, \beta_m)$ and $d_T=(\gamma_1,\gamma_2,\cdots, \gamma_l)$ the diagonals of the matrix $R,S$ and $T$ respectively. Suppose that $n,m$ and $l$ are odd and we have $\left|\sum^{n}_{i=1} \alpha_i\right|=a$, $\left|\sum^{m}_{j=1} \beta_j\right|=b$ and $\left|\sum^{l}_{h=1} \gamma_h\right|=c$, then, in $d_R$:
\begin{itemize}
    \item [(i)] $d_R$  consist either of $\frac{n+a}{2}$ elements equal to $1$ and $\frac{n-a}{2}$ elements equal to $-1$, or
    \item[(ii)] $\frac{n-a}{2}$ elements equal to $1$ and $\frac{n+a}{2}$ elements equal to $-1$.
\end{itemize}
In $d_S$ we have:
\begin{itemize}
    \item [(iii)] Consist of $\frac{m+b}{2}$ elements equal to $1$ and $\frac{m-b}{2}$ elements equal to $-1$, or
    \item[(iv)] Consist of$\frac{m-b}{2}$ elements equal to $1$ and $\frac{m+b}{2}$ elements equal to $-1$.
\end{itemize}
and, in $d_T$ we have that:
\begin{itemize}
    \item [(v)] Consist of $\frac{l+c}{2}$ elements equal to $1$ and $\frac{l-c}{2}$ elements equal to $-1$, or
    \item[(vi)] Consist of $\frac{l-c}{2}$ elements equal to $1$ and
    $\frac{l+c}{2}$ elements equal to $-1$.
\end{itemize}

Remember that there are $\Delta (d_R)$, $\Delta(d_S)$ and $\Delta(d_T)$ independent entries in the sub-matrices $R,S$ and $T$ respectively.Then we have eight situations

\begin{itemize}
\item[(a)] Consider (i)$\wedge$(iii)$\wedge$(v) then we have the matrix representation:

$$\begin{array}{ccccc}& & &  \overbrace{\begin{tabular}{|c|c|c|}\hline $d_R$ & $d_S$ & $d_T$ \\ \hline  $\frac{n-a}{2}$ & $\frac{m-b}{2}$ & $\frac{l-c}{2}$ \\ \hline  \end{tabular}}^{Number \ of \ -1 \ in:} \\  & &  & & \\   & \overbrace{\begin{tabular}{|c|c|} \hline $d_R$ & $\frac{n+a}{2}$ \\ \hline $d_S$ & $\frac{m+b}{2}$ \\  \hline $d_T$ & $\frac{l+c}{2}$ \\ \hline \end{tabular}}^{Number   \ of \  1 \ in:}  &  &  \left[\begin{array}{c|c|c} \Delta(d_R) & (\frac{n+a}{2})(\frac{m-b}{2})+(\frac{n-a}{2})(\frac{m+b}{2}) & (\frac{n+a}{2})(\frac{l-c}{2})+(\frac{n-a}{2})(\frac{l+c}{2})   \\ \hline 0 & \Delta(d_S) & 0 \\ \hline 0 & 0 & \Delta(d_T)\end{array}\right]\end{array}$$
then we have

$$\binom{n}{\frac{n-a}{2}}\binom{m}{\frac{m-b}{2}}\binom{l}{\frac{l-c}{2}} q^{\Delta(d_R)+\Delta(d_S)+\Delta(d_T)+ \frac{1}{2}(nm-ab)+ \frac{1}{2}(nl-ac)},$$ involutions in this case.

\item[(b)] If we consider (i)$\wedge$(iii)$\wedge$(vi) then, in the matrix representation:

$$\begin{array}{ccccc}& & &  \overbrace{\begin{tabular}{|c|c|c|} \hline $d_R$ & $d_S$ & $d_T$ \\ \hline $\frac{n-a}{2}$ & $\frac{m-b}{2}$ & $\frac{l+c}{2}$ \\ \hline \end{tabular}}^{Number \ of \ -1 \ in: } \\  & &  & & \\   & \overbrace{\begin{tabular}{|c|c|c|}\hline $d_T$ & $\frac{n+a}{2}$ \\ \hline $d_S$& $\frac{m+b}{2}$ \\  \hline $d_T$ & $\frac{l-c}{2}$ \\ \hline \end{tabular}}^{Number   \ of \  1 \ in :}  &  &  \left[\begin{array}{c|c|c} \Delta(d_R) & (\frac{n+a}{2})(\frac{m-b}{2})+(\frac{n-a}{2})(\frac{m+b}{2}) & (\frac{n+a}{2})(\frac{l+c}{2})+(\frac{n-a}{2})(\frac{l-c}{2})   \\ \hline 0 & \Delta(d_S) & 0 \\ \hline 0 & 0 & \Delta(d_T)\end{array}\right]\end{array}$$
then in this case we have

$$\binom{n}{\frac{n-a}{2}}\binom{m}{\frac{m-b}{2}}\binom{l}{\frac{l-c}{2}} q^{\Delta(d_R)+\Delta(d_S)+\Delta(d_T)+ \frac{1}{2}(nm-ab)+ \frac{1}{2}(nl+ac)},$$ involutions.

\item[(c)] If (i)$\wedge$(iv)$\wedge$(v) we have:

$$\begin{array}{ccccc}& & &  \overbrace{\begin{tabular}{|c|c|c|}\hline $d_T$ & $d_S$ & $d_T$ \\ \hline  $\frac{n-a}{2}$ & $\frac{m+b}{2}$ & $\frac{l-c}{2}$ \\ \hline\end{tabular}}^{Number \ of \ -1:} \\  & &  & & \\   & \overbrace{\begin{tabular}{|c|c|} \hline $d_T$ & $ \frac{n+a}{2}$ \\ \hline  $d_S$ & $\frac{m-b}{2}$ \\  \hline $d_T$ & $\frac{l+c}{2}$ \\ \hline  \end{tabular}}^{Number   \ of \  1 \ in:}  &  &  \left[\begin{array}{c|c|c} \Delta(d_R) & (\frac{n+a}{2})(\frac{m+b}{2})+(\frac{n-a}{2})(\frac{m-b}{2}) & (\frac{n+a}{2})(\frac{l-c}{2})+(\frac{n-a}{2})(\frac{l+c}{2})   \\ \hline 0 & \Delta(d_S) & 0 \\ \hline 0 & 0 & \Delta(d_T)\end{array}\right]\end{array}$$
Here, we have

$$\binom{n}{\frac{n-a}{2}}\binom{m}{\frac{m-b}{2}}\binom{l}{\frac{l-c}{2}} q^{\Delta(d_R)+\Delta(d_S)+\Delta(d_T)+ \frac{1}{2}(nm+ab)+ \frac{1}{2}(nl-ac)},$$ involutions.

\item[(d)] If (i)$\wedge$(iv)$\wedge$(vi) we have

$$\begin{array}{ccccc}& & &  \overbrace{\begin{tabular}{|c|c|c|} \hline $d_R$ & $d_S$ & $d_T$ \\ \hline $ \frac{n-a}{2}$  & $\frac{m+b}{2}$ &$ \frac{l+c}{2}$ \\ \hline \end{tabular}}^{Number \ of \ -1} \\  & &  & & \\   & \overbrace{\begin{tabular}{|c|c|}  \hline $d_R$ & $\frac{n+a}{2}$ \\ \hline $d_S$ & $\frac{m-b}{2}$ \\  \hline $d_T$ & $\frac{l-c}{2}$ \\ \hline\end{tabular}}^{Number   \ of \  1:}  &  &  \left[\begin{array}{c|c|c} \Delta(d_R) & (\frac{n+a}{2})(\frac{m+b}{2})+(\frac{n-a}{2})(\frac{m-b}{2}) & (\frac{n+a}{2})(\frac{l+c}{2})+(\frac{n-a}{2})(\frac{l-c}{2})   \\ \hline 0 & \Delta(d_S) & 0 \\ \hline 0 & 0 & \Delta(d_T)\end{array}\right]\end{array}$$
then we have

$$\binom{n}{\frac{n-a}{2}}\binom{m}{\frac{m-b}{2}}\binom{l}{\frac{l-c}{2}} q^{\Delta(d_R)+\Delta(d_S)+\Delta(d_T)+ \frac{1}{2}(nm+ab)+ \frac{1}{2}(nl+ac)},$$ involutions in this case.

\item[(e)] If (ii)$\wedge$(iii)$\wedge$(v) we have
$$\begin{array}{ccccc}& & &  \overbrace{\begin{tabular}{|c|c|c|} \hline $d_R$ & $d_S$ & $d_T$ \\ \hline $\frac{n+a}{2}$ & $\frac{m-b}{2}$ & $\frac{l-c}{2}$ \\ \hline \end{tabular}}^{Number \ of \ -1 \ in:} \\  & &  & & \\   & \overbrace{\begin{tabular}{|c|c|}\hline $d_R$ & $\frac{n-a}{2}$ \\ \hline $d_S$ & $\frac{m+b}{2}$ \\  \hline $d_T$ & $\frac{l+c}{2}$ \\ \hline \end{tabular}}^{Number   \ of \  1 \ in:}  &  &  \left[\begin{array}{c|c|c} \Delta(d_R) & (\frac{n-a}{2})(\frac{m-b}{2})+(\frac{n+a}{2})(\frac{m+b}{2}) & (\frac{n-a}{2})(\frac{l-c}{2})+(\frac{n+a}{2})(\frac{l+c}{2})   \\ \hline 0 & \Delta(d_S) & 0 \\ \hline 0 & 0 & \Delta(d_T)\end{array}\right]\end{array}$$
then we have

$$\binom{n}{\frac{n-a}{2}}\binom{m}{\frac{m-b}{2}}\binom{l}{\frac{l-c}{2}} q^{\Delta(d_R)+\Delta(d_S)+\Delta(d_T)+ \frac{1}{2}(nm+ab)+ \frac{1}{2}(nl+ac)}$$ involutions.

\item[(f)] If (ii)$\wedge$(iii)$\wedge$(vi) we have

$$\begin{array}{ccccc}& & &  \overbrace{\begin{tabular}{|c|c|c|}\hline $d_R$ & $d_S$ & $d_T$ \\ \hline $\frac{n+a}{2}$ & $\frac{m-b}{2}$ & $\frac{l+c}{2}$ \\ \hline \end{tabular}}^{Number \ of \ -1:} \\  & &  & & \\   & \overbrace{\begin{tabular}{|c|c|} \hline $d_T$ & $ \frac{n-a}{2}$ \\ \hline $d_S$ & $\frac{m+b}{2}$ \\  \hline $d_T$ & $\frac{l-c}{2}$ \\ \hline  \end{tabular}}^{Number   \ of \  1 :}  &  &  \left[\begin{array}{c|c|c} \Delta(d_R) & (\frac{n-a}{2})(\frac{m-b}{2})+(\frac{n+a}{2})(\frac{m+b}{2}) & (\frac{n-a}{2})(\frac{l+c}{2})+(\frac{n+a}{2})(\frac{l-c}{2})   \\ \hline 0 & \Delta(d_S) & 0 \\ \hline 0 & 0 & \Delta(d_T)\end{array}\right]\end{array}$$
then we have

$$\binom{n}{\frac{n-a}{2}}\binom{m}{\frac{m-b}{2}}\binom{l}{\frac{l-c}{2}} q^{\Delta(d_R)+\Delta(d_S)+\Delta(d_T)+ \frac{1}{2}(nm+ab)+ \frac{1}{2}(nl-ac)},$$ involutions here.

\item[(g)] If (ii)$\wedge$(iv)$\wedge$(v) we have

$$\begin{array}{ccccc}& & &  \overbrace{\begin{tabular}{|c|c|c|}\hline $d_R$ & $d_S$ & $d_T$ \\ \hline $ \frac{n+a}{2}$ & $\frac{m+b}{2}$ & $\frac{l-c}{2}$ \\ \hline  \end{tabular}}^{Number \ of \ -1 \ in: } \\  & &  & & \\   & \overbrace{\begin{tabular}{|c|c|} \hline $d_R$ & $ \frac{n-a}{2}$ \\ \hline $d_S$ & $\frac{m-b}{2}$ \\  \hline $d_T$ & $ \frac{l+c}{2}$ \\ \hline \end{tabular}}^{Number   \ of \  1 in: }  &  &  \left[\begin{array}{c|c|c} \Delta(d_R) & (\frac{n-a}{2})(\frac{m+b}{2})+(\frac{n+a}{2})(\frac{m-b}{2}) & (\frac{n-a}{2})(\frac{l-c}{2})+(\frac{n+a}{2})(\frac{l+c}{2})   \\ \hline 0 & \Delta(d_S) & 0 \\ \hline 0 & 0 & \Delta(d_T)\end{array}\right]\end{array}$$
Here we have:

$$\binom{n}{\frac{n-a}{2}}\binom{m}{\frac{m-b}{2}}\binom{l}{\frac{l-c}{2}} q^{\Delta(d_R)+\Delta(d_S)+\Delta(d_T)+ \frac{1}{2}(nm-ab)+ \frac{1}{2}(nl+ac)},$$ involutions.

\item[(h)] Finally, if (ii)$\wedge$(iv)$\wedge$(vi) we have:
$$\begin{array}{c p{0,1cm} c} & &  \overbrace{\begin{tabular}{|c|c|c|} \hline $d_R$ & $d_S$ & $d_T$ \\ \hline $ \frac{n+a}{2}$ & $\frac{m+b}{2}$ & $ \frac{l+c}{2}$ \\ \hline  \end{tabular}}^{Number \ of \ -1 \ in:} \\  & &    \\
 \overbrace{\begin{tabular}{|c|c|} \hline $d_R$ & $ \frac{n-a}{2}$ \\  \hline $d_S$ & $\frac{m-b}{2}$ \\  \hline $d_T$ & $\frac{l-c}{2}$ \\ \hline \end{tabular}}^{Number   \ of \  1 \ in : }  &  &  \left[\begin{array}{c|c|c} \Delta(d_R) & (\frac{n-a}{2})(\frac{m+b}{2})+(\frac{n+a}{2})(\frac{m-b}{2}) & (\frac{n-a}{2})(\frac{l+c}{2})+(\frac{n+a}{2})(\frac{l-c}{2})   \\ \hline 0 & \Delta(d_S) & 0 \\ \hline 0 & 0 & \Delta(d_T)\end{array}\right]\end{array}$$

then we have

$$\binom{n}{\frac{n-a}{2}}\binom{m}{\frac{m-b}{2}}\binom{l}{\frac{l-c}{2}} q^{\Delta(d_R)+\Delta(d_S)+\Delta(d_T)+ \frac{1}{2}(nm-ab)+ \frac{1}{2}(nl-ac)},$$ involutions.
\end{itemize}

Combining all situations we have:

$$\mathcal{P}(Y,\mathbb{K})=
\sum^{n}_{\substack{1\leq a \leq n\\ 2| (n-a)}  }\sum^{m}_{\substack{1\leq b \leq m\\ 2| (m-b)}}\sum^{l}_{\substack{1\leq c \leq l \\ 2|(l-c)}}\binom{ n}{ \frac{n-a}{2}}\binom{m}{ \frac{m-b}{2}}\binom{l}{ \frac{l-c}{2}}q^{\Delta(d_R)+\Delta(d_S)+\Delta(d_T)}\cdot \mathcal{F}(a,b,c),$$
with
$$\mathcal{F}(a,b,c)=2q^{\frac{1}{2}n(m+l)}\cdot \left\{q^{\frac{1}{2}a(b+c)}+q^{\frac{1}{2}a(b-c)}+q^{\frac{1}{2}[-a(b+c)]}+q^{\frac{1}{2}[-a(b-c)]}\right\}.
$$

In the case, for example, that $n$ is even and $m$ and $l$ are odd then  and we can consider the followings situations:

\begin{itemize}
    \item [(A)] There are $\frac{n}{2}$ elements $1$  in $d_R$.
\end{itemize}

And, we have considered the combinations:
(A)$\wedge$(iii)$\wedge$(v), (A)$\wedge$(iii)$\wedge$(vi), (A)$\wedge$(iv)$\wedge$(v) and (A)$\wedge$(iv)$\wedge$(v), for example in the first case we have the following matrix form:

$$\begin{array}{ccccc}& & &  \overbrace{\begin{tabular}{|c|c|c|}\hline
  $d_R$  & $d_S$ & $d_T$ \\ \hline
   $\frac{n}{2}$  & $\frac{m-b}{2}$ & $\frac{l-c}{2}$ \\ \hline
\end{tabular}}^{Number \ of \ -1 \ in } \\  & &  & & \\   & \overbrace{\begin{tabular}{|c|c|}\hline
  $d_R$  & $\frac{n}{2}$ \\ \hline  $d_S$ & $\frac{m+b}{2}$ \\ \hline
   $d_T$  & $\frac{l+c}{2}$ \\ \hline
\end{tabular}}^{Number   \ of \  1 \ in:}  &  &  \left[\begin{array}{c|c|c} \Delta(d_R) & (\frac{n}{2})(\frac{m-b}{2})+(\frac{n}{2})(\frac{m+b}{2}) & (\frac{n}{2})(\frac{l-c}{2})+(\frac{n}{2})(\frac{l+c}{2})   \\ \hline 0 & \Delta(d_S) & 0 \\ \hline 0 & 0 & \Delta(d_T)\end{array}\right]\end{array}$$

and sumarizing the four combinations we conclude that we have that add the term

$$ \displaystyle\sum_{\substack{1\leq b \leq m \\ 2|(m-b)}} \ \sum_{\substack{1\leq c \leq l \\ 2|(l-c)}}\binom{n}{\frac{n}{2}}\binom{m}{\frac{m-b}{2}}\binom{l}{\frac{l-c}{2}}q^{\left\{\frac{1}{4}\left(n^2+m^2+l^2-b^2-c^2\right)\right\}}\cdot \left\{4 q^{\frac{n(m+l)}{2}}\right\} $$

Similarlly, if $m$ is even and $n$ and $l$ are odd, we have that consider 

\begin{itemize}
    \item [(B)] There are $\frac{m}{2}$ elements $1$ in $d_S$.
\end{itemize}
Then we have the following combinations: (i)$\wedge$(B)$\wedge$(v),
(i)$\wedge$(B)$\wedge$(vi),
(ii)$\wedge$(B)$\wedge$(v) and
(ii)$\wedge$(B)$\wedge$(vi). In this situation we have add the term

$$\displaystyle\sum_{\substack{1\leq a \leq n \\ 2|(n-a)}} \ \sum_{\substack{1\leq c \leq l \\ 2|(l-c)}}\binom{n}{\frac{n-a}{2}}\binom{m}{\frac{m}{2}}\binom{l}{\frac{l-c}{2}}q^{\left\{\frac{1}{4}\left(n^2+m^2+l^2-a^2-c^2\right)\right\}}\cdot 2q^{\frac{nm}{2}}\cdot \left\{ q^{\frac{nl-ac}{2}}+q^{\frac{nl+ac}{2}}\right\}.  $$

And, if $l$ is even and $n$ and $m$ are odd numbers, we have that consider:

\begin{itemize}
    \item [(C)] There are $\frac{l}{2}$ elements $1$ in $d_T$.
\end{itemize}
In this case, we have  the following combinations 
(i)$\wedge$(iii)$\wedge$(C),
(i)$\wedge$(iv)$\wedge$(C),
(ii)$\wedge$(iii)$\wedge$(C) and
(ii)$\wedge$(iv)$\wedge$(C), here we have add the term:

$$\displaystyle\sum_{\substack{1\leq a \leq n \\ 2|(n-a)}} \ \sum_{\substack{1\leq b \leq m \\ 2|(m-b)}}\binom{n}{\frac{n-a}{2}}\binom{m}{\frac{m-b}{2}}\binom{l}{\frac{l}{2}}q^{\left\{\frac{1}{4}\left(n^2+m^2+l^2-a^2-b^2\right)\right\}}\cdot 2q^{\frac{ln}{2}}\cdot \left\{ q^{\frac{nm-ab}{2}}+q^{\frac{nm+ab}{2}}\right\} . $$

In the case that $n$ and $m$ are even numbers and $l$ is odd, we have consider the combinations:
(A)$\wedge$(iii)$\wedge$(v),
(A)$\wedge$(iii)$\wedge$(vi),
(A)$\wedge$(iv)$\wedge$(v),
(A)$\wedge$(iv)$\wedge$(vi),
(i)$\wedge$(B)$\wedge$(v),
(ii)$\wedge$ (B) $\wedge$(v),
(i)$\wedge$(B)$\wedge$(vi),
(ii)$\wedge$(B)$\wedge$(vi),
(A)$\wedge$(B)$\wedge$(v) and
(A)$\wedge$ (B)  $\wedge$(vi). We can observe that the four first combinations are similar to the case that $n$ is only even, and the following four combinations are similar to the case that $m$ is only even. The last two cases are calculated via matrix form, for example in the case that (A)$\wedge$(B)$\wedge$(v) we have:

   $$ \begin{array}{ccccc}& & &  \overbrace{\begin{tabular}{|c|c|c|}\hline
  $d_R$  & $d_S$ & $d_T$ \\ \hline
   $\frac{n}{2}$  & $\frac{m}{2}$ & $\frac{l-c}{2}$ \\ \hline
\end{tabular}}^{Number \ of \ -1 \ in } \\  & &  & & \\   & \overbrace{\begin{tabular}{|c|c|}\hline
  $d_R$  & $\frac{n}{2}$ \\ \hline  $d_S$ & $\frac{m}{2}$ \\ \hline
   $d_T$  & $\frac{l+c}{2}$ \\ \hline
\end{tabular}}^{Number   \ of \  1 \ in:}  &  &  \left[\begin{array}{c|c|c} \Delta(d_R) & \frac{nm}{2} & \frac{nl}{2}   \\ \hline 0 & \Delta(d_S) & 0 \\ \hline 0 & 0 & \Delta(d_T)\end{array}\right]\end{array}
$$

and if $(A)\wedge (B) \wedge(vi)$ we have:
  $$ \begin{array}{ccccc}& & &  \overbrace{\begin{tabular}{|c|c|c|}\hline
  $d_R$  & $d_S$ & $d_T$ \\ \hline
   $\frac{n}{2}$  & $\frac{m}{2}$ & $\frac{l+c}{2}$ \\ \hline
\end{tabular}}^{Number \ of \ -1 \ in } \\  & &  & & \\   & \overbrace{\begin{tabular}{|c|c|}\hline
  $d_R$  & $\frac{n}{2}$ \\ \hline  $d_S$ & $\frac{m}{2}$ \\ \hline
   $d_T$  & $\frac{l-c}{2}$ \\ \hline
\end{tabular}}^{Number   \ of \  1 \ in:}  &  &  \left[\begin{array}{c|c|c} \Delta(d_R) & \frac{nm}{2} & \frac{nl}{2}   \\ \hline 0 & \Delta(d_S) & 0 \\ \hline 0 & 0 & \Delta(d_T)\end{array}\right]\end{array}
  $$  

Finally, in this case we have add the terms
$$\begin{array}{l}
\displaystyle\sum_{\substack{1\leq b \leq m \\ 2|(m-b)}} \ \sum_{\substack{1\leq c \leq l \\ 2|(l-c)}}\binom{n}{\frac{n}{2}}\binom{m}{\frac{m-b}{2}}\binom{l}{\frac{l-c}{2}}q^{\left\{\frac{1}{4}\left(n^2+m^2+l^2-b^2-c^2\right)\right\}}\cdot \left\{4 q^{\frac{n(m+l)}{2}}\right\} \\
+
    \displaystyle\sum_{\substack{1\leq a \leq n \\ 2|(n-a)}} \ \sum_{\substack{1\leq c \leq l \\ 2|(l-c)}}\binom{n}{\frac{n-a}{2}}\binom{m}{\frac{m}{2}}\binom{l}{\frac{l-c}{2}}q^{\left\{\frac{1}{4}\left(n^2+m^2+l^2-a^2-c^2\right)\right\}}\cdot 2q^{\frac{nm}{2}}\cdot \left\{ q^{\frac{nl-ac}{2}}+q^{\frac{nl+ac}{2}}\right\}    
    \\
    + 
    \displaystyle\sum_{\substack{1\leq a \leq n \\ 2|(n-a)}} \ \sum_{\substack{1\leq b \leq m \\ 2|(m-b)}}\binom{n}{\frac{n-a}{2}}\binom{m}{\frac{m-b}{2}}\binom{l}{\frac{l}{2}}q^{\left\{\frac{1}{4}\left(n^2+m^2+l^2-a^2-b^2\right)\right\}}\cdot 2q^{\frac{ln}{2}}\cdot \left\{ q^{\frac{nm-ab}{2}}+q^{\frac{nm+ab}{2}}\right\} \\    
    +
    
    \displaystyle\sum_{\substack{1\leq c \leq l \\ 2|(l-c)}} \ \binom{n}{\frac{n}{2}}\binom{m}{\frac{m}{2}}\binom{l}{\frac{l-c}{2}}q^{\left\{\frac{1}{4}\left(n^2+m^2+l^2-c^2\right)\right\}}\cdot  q^{\frac{n(m+l)}{2}}.
    \end{array}
    $$
    
The cases that $n$ and $l$ are odd numbers or $m$ and $l$ are odd numbers are similar.
Finally, in the case that all numbers are even, we have consider all combinations in each situations, only the last situation is when 
\begin{itemize}
    \item [(D)] There are $\frac{n}{2}$ elements $1$ in $d_R$, $\frac{m}{2}$ elements $1$ in $d_S$ and $\frac{l}{2}$ elements $1$ in $d_S$.
\end{itemize}
in the matrix form:

$$\begin{array}{ccccc}& & &  \overbrace{\begin{tabular}{|c|c|c|}\hline
  $d_R$  & $d_S$ & $d_T$ \\ \hline
   $\frac{n}{2}$  & $\frac{m}{2}$ & $\frac{l}{2}$ \\ \hline
\end{tabular}}^{Number \ of \ -1 \ in } \\  & &  & & \\   & \overbrace{\begin{tabular}{|c|c|}\hline
  $d_R$  & $\frac{n}{2}$ \\ \hline  $d_S$ & $\frac{m}{2}$ \\ \hline
   $d_T$  & $\frac{l}{2}$ \\ \hline
\end{tabular}}^{Number   \ of \  1 \ in:}  &  &  \left[\begin{array}{c|c|c} \Delta(d_R) & \frac{nm}{2} & \frac{nl}{2}   \\ \hline 0 & \Delta(d_S) & 0 \\ \hline 0 & 0 & \Delta(d_T)\end{array}\right]\end{array}$$
and we have add the all previous sums and the term

$$\cdots + \displaystyle \binom{n}{\frac{n}{2}}\binom{m}{\frac{m}{2}}\binom{l}{\frac{l}{2}}q^{\left\{\frac{1}{4}\left(n^2+m^2+l^2\right)\right\}}\cdot  q^{\frac{n(m+l)}{2}}. $$
\end{proof}


\section{Tables presentations}
In the Table 2, we present the values of $\mathcal{P}(n+m+1,\mathbb{K})$ for $1\leq m \leq 14$. Remember that the values for $\mathcal{I}(n+1,\mathbb{K})$ are calculated in Slowik \cite{slowik1}. In the Table 3, we present the number of involutions on the Rhombuses poset for some   values $m,n$ such that $1 \leq n \leq 5$ and $1 \leq m \leq 5$ and finally, in the Table 4 present the number of involutions in the $Y$ poset for some values of $n,m$ and $l$ such that $1\leq n \leq 3$, $1\leq m \leq 3$ and $1 \leq l \leq 3$.

\begin{table}[H]
\caption{Numbers of involutions the $\mathcal{P}(n+1+m,\mathbb{K})$  in $I(X,\mathbb{K})$ for $|\mathbb{K}|=q$, $0\leq m\leq 14$.}
\centering
\begin{tabular}{ccl}\hline\hline
$m$ &  &  \multicolumn{1}{c}{ $\mathcal{P}(n+m+1,\mathbb{K})$}\\ [0.1ex]
\hline
1 & & $\mathcal{I}(n+1,\mathbb{K})\times (q+1)$\\ 
2 &  &  $\mathcal{I}(n+1,\mathbb{K})\times (3q^2+1)$\\ 
3 &  & $\mathcal{I}(n+1,\mathbb{K})\times (3q^4+4q^3+1)$\\ 
4 &  & $\mathcal{I}(n+1,\mathbb{K})\times (10q^6+5q^4+1)$\\ 
5 &  & $\mathcal{I}(n+1,\mathbb{K})\times (10q^9+15q^8+6q^5+1)$\\ 
6 &  & $\mathcal{I}(n+1,\mathbb{K})\times (35q^{12}+21q^{10}+7q^6+1)$\\ 
7 &  &  $\mathcal{I}(n+1,\mathbb{K})\times (35q^{16}+56q^{15}+28q^{12}+8q^7+1)$\\ 
8 &  & $\mathcal{I}(n+1,\mathbb{K})\times (126q^{20}+84q^{18}+36q^{14}+9q^8+1)$ \\
9 &  & $\mathcal{I}(n+1,\mathbb{K})\times (126q^{25}+210q^{24}+120q^{21}+45q^{16}+10q^9+1)$  \\ 
10 & &  $\mathcal{I}(n+1,\mathbb{K})\times (462q^{30}+330q^{28}+165q^{24}+55q^{18}+11q^{10}+1)$  \\  
11 &  &  $\mathcal{I}(n+1,\mathbb{K})\times (462q^{36}+792q^{35}+495q^{32}+220q^{27}+66q^{20}+12q^{11}+1)$  \\ 
12 &  &  $\mathcal{I}(n+1,\mathbb{K})\times (1716q^{42}+1224q^{40}+715q^{36}+286q^{30}+78q^{22}+13q^{12}+1)$  \\ 
13 &  &  $\mathcal{I}(n+1,\mathbb{K})\times (1716q^{49}+3003q^{48}+2002q^{45}+1001q^{40}+364q^{33}+91q^{24}+ $\\
   &  &  $14q^{13}+1)$  \\ 
14 &  &  $\mathcal{I}(n+1,\mathbb{K})\times (6435q^{56}+5005q^{54}+3003q^{50}+1365q^{44}+455q^{36}+105q^{26}+ $ \\ 
   &  &  $15q^{14}+1)$  \\ [1ex]
\hline
\end{tabular}
\label{}
\end{table}

\begin{table}[H]
\caption{Numbers of involutions the Rhombuses Poset.}
\centering
\begin{tabular}{ccl}\hline\hline
$n$ & $m$  &  \multicolumn{1}{c}{Rhombuses Poset} \\ [0.1ex]
\hline
1 & 1  & $2q^4+8q^3+4q^2+1$\\
1 & 2  & $2q^6+12q^5+10q^4+4q^3+2q^2+2 $\\ 
1 & 3  & $8q^8+24q^7+14q^6+8q^5+6q^4+2q^2+2 $\\ 
1 & 4  & $8q^{11}+38q^{10}+40q^{9}+14q^8+8q^7+8q^6+8q^5+2q^2+2 $\\ 
1& 5  & $30q^{14}+80q^{13}+52q^{12}+40q^{11}+22q^{10}+10q^8 +8q^7+10q^6+2q^2+2$\\ 
2 & 1 & $2q^6+12q^5+10q^4+4q^3+2q^2+2$\\ 
2& 2  & $2q^8+16q^7+24q^6+8q^5+4q^4+8q^3+2$\\ 
2& 3  & $8q^{10}+40q^9+30q^8+20q^7+16q^6+8q^4+2 $\\ 
2&  4 & $8q^{13}+54q^{12}+84q^{11}+32q^{10}+16q^9+38q^8+8q^7+8q^5+2q^4+4q^3+2 $\\ 
2 & 5   & $30q^{16}+140q^{15}+102q^{14}+84q^{13}+70q^{12}+32q^{10}+28q^{9}+8q^{8}+10q^6  $\\
  &     & $+2q^4+4q^3+2$  \\ 
3 &1    & $8q^8+24q^7+14q^6+8q^5+6q^4+2q^2+2 $\\ 
3& 2  & $8q^{10}+40q^9+30q^8+20q^7+16q^6+8q^4+2 $\\ 
3&3 & $32q^{12}+72q^{11}+48q^{10}+48q^9+18q^8+8q^7+16q^6+12q^4+2 $ \\ 
3&  4  &$ 32q^{15}+126q^{14}+120q^{13}+66q^{12}+64q^{11}+26q^{10}+3
2q^9+20q^8+8q^6+ $ \\ 
  &    & $ 8q^5+8q^4+2 $\\ 
3&  5  &$ 120q^{18}+240q^{17}+178q^{16}+200q^{15}+66q^{14}+40q^{13}+70q^{12}+24q^{11}+ $ \\ 
  &    & $52q^{10}+8q^9+$\\
 & & $18q^6+6q^4+2 $\\ 
4 & 1  & $8q^{11}+38q^{10}+40q^{9}+14q^8+8q^7+8q^6+8q^5+2q^2+2 $\\ 
4&  2  &$8q^{13}+54q^{12}+84q^{11}+32q^{10}+16q^9+38q^8+8q^7+8q^5+2q^4+4q^3+2 $\\ 
4&   3 &$ 32q^{15}+126q^{14}+120q^{13}+66q^{12}+64q^{11}+26q^{10}+32q^9+20q^8+8q^6+$ \\
 &     & $ 8q^5+8q^4+2 $\\ 
4&   4 &$ 32q^{18}+184q^{17}+290q^{16}+128q^{15}+64q^{14}+160q^{13}+64q^{12}+32q^{10}+$ \\
 &     & $ 24q^9+28q^8+8q^5+2 $\\ 
4&   5 &$ 120q^{21}+450q^{20}+408q^{19}+274q^{18}+280q^{17}+80q^{16}+128q^{15}+134q^{14}+$ \\
 &     & $32q^{13}+30q^{12}+$\\
 & & $+40q^{11}+30q^{10}+8q^9+14q^8+10q^6+8q^5+2$\\ 
5& 1 &$ 30q^{14}+80q^{13}+52q^{12}+40q^{11}+22q^{10}+10q^8 +8q^7+10q^6+2q^2+2$ \\ 
5 & 2   & $30q^{16}+140q^{15}+102q^{14}+84q^{13}+70q^{12}+32q^{10}+28q^{9}+8q^{8}+10q^6+$ \\
 &      & $ 2q^4+4q^3+2$  \\ 
5&  3  &$ 120q^{18}+240q^{17}+178q^{16}+200q^{15}+66q^{14}+40q^{13}+70q^{12}+24q^{11}+$ \\
 &      & $52q^{10}+8q^9+18q^6+6q^4+2 $\\ 
5&   4 &$ 120q^{21}+450q^{20}+408q^{19}+274q^{18}+280q^{17}+80q^{16}+128q^{15}+134q^{14}+$ \\
 &     & $32q^{13}+30q^{12}+40q^{11}+30q^{10}+8q^9+14q^8+10q^6+8q^5+2 $\\
5&   5 &$ 480q^{24}+840q^{23}+620q^{22}+840q^{21}+202q^{20}+200q^{19}+460q^{18}+220q^{16}+$ \\
 &     & $ 8q^{15}+40q^{13}+70q^{12}+48q^{11}+4q^{10}+20q^6+2 $\\ [1ex]
\hline
\end{tabular}
\label{}
\end{table}

\begin{table}[H]
\caption{Numbers of involutions the $Y$ Poset.}
\centering
\begin{tabular}{cccl} \hline\hline
$n$ & $m$ & $l$ & \multicolumn{1}{c}{$Y$ Poset} \\ [0.1ex]
\hline
1 & 1 & 1 & $2q^2+4q+2$\\ 
1&  1&  2&$ 6q^3+6q^2+2q+2 $\\ 
1&  2&  1&$ 6q^3+6q^2+2q+2 $\\ 
1&  2&  2&$ 18q^4+12q^2+2 $\\ 
1 & 1 & 3 & $6q^5+14q^4+8q^3+2q+2$\\ 
1 & 3 & 1 & $6q^5+14q^4+8q^3+2q+2)$\\ 
2& 1 & 1 & $8q^3 +2 q^4+4q^2+2 $\\ 
2& 1 & 2 & $ 2q^6+12q^5+10q^4+4q^3+2q^2+2$\\
2& 2 &  1& $ 2q^6+12q^5+10q^4+4q^3+2q^2+2$\\ 
2& 2&  2&$ 2q^8+16q^7+24q^6+8q^5+4q^4+8q^3+2 $\\ 
3 & 1 & 1 & $8q^6+12q^5+6q^4+4q^3+2)$\\ 
3& 1 & 3 & $2q^{12}+24q^{11}+42q^{10}+26q^9+12q^8+6q^7+6q^6+6q^5+2q^3+2)$\\
3 &3  &1  & $2q^{12}+24q^{11}+42q^{10}+26q^9+12q^8+6q^7+6q^6+6q^5+2q^3+2)$\\ 
3& 3 & 3 & $2q^{18}+12q^{17}+72q^{16}+144q^{15}+108q^{14}+36q^{13}+36q^{12}+48q^{11}+18q^{10}$ \\
 &   &   & $+4q^9+18q^8+12q^5+2)$\\ [1ex]
\hline
\end{tabular}
\label{}
\end{table}

\section{Conclusions}
By the lemma 1.18 of Gubareni-Hazewinkel \cite{gub1}, we can calculated the number of all involutions on an incidence algebra of an finite poset, but an unique formula to calculate this is complicated by the multiples forms of an finite poset. Fortunately with this article we shown an road to can calculate this.

\bigskip
\end{document}